\documentclass[12pt,reqno]{article}
\usepackage{amsfonts}
\usepackage{graphicx}
\usepackage{epstopdf}
\usepackage{amsmath}
\usepackage{amsthm}
\usepackage{latexsym,bm}
\usepackage{amssymb}
\usepackage{indentfirst}
\usepackage{amsfonts}
\usepackage{graphicx}
\usepackage{epstopdf}
\usepackage{amsmath}
\usepackage{amsthm}
\usepackage{latexsym,bm}
\usepackage{amssymb}
\usepackage{indentfirst}

\usepackage{setspace}
\newtheorem{thm}{Theorem}[section]
\newtheorem{cor}[thm]{Corollary}
\newtheorem{lem}[thm]{Lemma}

\newtheorem{rem}[thm]{Remark}
\newtheorem{exm}{Example}

\numberwithin{equation}{section}
\usepackage{graphicx}

\begin{document}
\title{{\bf 
Conformal geometry, Euler numbers, and
global 
invertibility 
in higher dimensions
}}
\author{{\bf Frederico  Xavier\thanks{Partial support provided by the John William and Helen Stubbs Potter Professorship.
}
}}

\date{}
\maketitle

\begin{quote}
%\vskip-35pt
\small {\bf Abstract}.   
\noindent  It is shown that  in dimension at least three  a local diffeomorphism of Euclidean $n$-space into itself is injective  provided that the  pull-back  of every plane is  a Riemannian submanifold  which  is conformal to a plane.  Using a similar technique one recovers the result that a polynomial local biholomorphism of complex $n$-space into itself is invertible   if and only if  the pull-back of every complex line is a connected rational curve. These results  are special cases  of  our main  theorem,  whose proof  uses  geometry,   complex analysis, elliptic partial differential equations, and topology.  
%As a refinement of the global invertibility problem, we address the question  of estimating the cardinality of a specific fiber $F^{-1}(q)$ of a locally invertible map in terms only of  objects naturally associated to $q$ itself.  Let $F:\mathbb R^n \to \mathbb R^n$  be  a local diffeomorphism,  $n>2$, and $q\in F(\mathbb R^n)$. As a special case of the main result, we  prove that $q$  is covered \it at most twice  \rm by $F$ if, for every plane $\pi$ that contains $q$,  
%$F^{-1}(\pi)$  is  conformal to $\mathbb C \mathbb P^1$ with $k_{\pi}< \infty $ points removed. When $k_{\pi}=1$ ($F^{-1}(\pi)\cong \mathbb C$) for every   
%such $\pi$, then $q$ is covered \it exactly once \rm  by $F$. Here, the complex structure of $F^{-1}(\pi)$ is associated to the  Riemannian metric induced by the standard inner product.   In particular, for  $n>2$ a local diffeomorphism $\mathbb R^n \to \mathbb R^n$  is   injective if the pull-backs of all affine planes are  Riemann surfaces conformal to $\mathbb C$. \rm  The  proofs  employ    arguments  from complex analysis, PDE,  geometry  and topology.  

\vskip1pt

\noindent \text{\bf 
AMS classification numbers: 14R15, 32H02, 53C21, 57R20 }
\end{quote}
\tableofcontents

\section{Introduction}

The   general question of deciding when a  locally invertible map  admits a smooth global  inverse is a recurring theme in mathematics,   about which  much remains to be understood,   even in  the finite dimensional case. It encapsulates neatly, in the  study of nonlinear systems  with as many variables as equations,  the questions of existence, uniqueness, and smooth dependence of the solution upon the  value in the target space. 

Given its wide scope, an all encompassing theory is not to be expected, but  along the years many intriguing  connections between global invertibility and  different fields (e.g algebra, algebraic and differential geometry, complex analysis, non-linear analysis, topology, dynamical systems, and mathematical economics) 
have  been unearthed. 
For a  small sample of works dealing with this problem,  where the search for new mechanisms of global invertibility is approached from a variety of points of view,   we refer the  reader to
\cite{AL},  \cite{B1} - \cite{HC}, \cite{CX} - \cite{GN}, \cite{AG} - \cite{GGJ}, \cite{FJ} - \cite {Ja}, \cite{BK}, \cite{GK}, \cite{MR} - \cite{NX2}, \cite{Pi} - \cite{X4}. 

In this paper we concentrate on the issue of  global  injectivity.  Our results, to be described shortly,  also establish a new connection between topology and conformal geometry. From a technical standpoint, the main tool is the theory of linear elliptic partial differential equations. 

As with many works on global invertibility,  the theorems in this paper  are of a ``foundational" nature, being part of  the overall effort of trying to understand the passage from local to global invertibility, in its many guises and for its own sake, without necessarily seeking to formulate  invertibility criteria that can be readily implemented. 

A historical example  showing   that  this point of view is a fruitful  one can be found  in  the pioneering works of J. Hadamard  \cite{JH} in 1904 and of H. Cartan  \cite{HC} in 1933. In modern language,  the Hadamard invertibility theorem states that a local diffeomorphism $F:X \to Y$ between Banach spaces is bijective if $\sup_{x\in X} ||DF(x)^{-1}||<\infty$  (see \cite{P}).    When $X=Y=\mathbb R^n$  
this analytic condition can be replaced by the  weaker requirement that all functionals $\langle F, v\rangle$, $v\neq 0$,  satisfy the Palais-Smale condition \cite{NX2}.

The above   condition in Hadamard's theorem is rather  strong, and this  limits its  use in applications. On the other hand, the  underlying ideas are of  great conceptual  value, as they are behind the  fundamental Cartan-Hadamard  theorem in differential geometry - the statement that the exponential map of a complete manifold of non-positive sectional curvatures is a covering map -, which captures the essence, in a global sense,  of non-positivity of curvature. 

A topological  result that subsumes   the  analytic result of  \cite{NX2} is the Balreira theorem \cite{B1},  stating that a local diffeomorphism  $F:\mathbb R^n \to \mathbb R^n$  is bijective if and only if the pre-image under $F$ of every affine hyperplane $H$ is non-empty and acyclic.

With the discovery of   subtler  manifestations of  the global invertibility phenomenon - a slow but steady  process -, increasingly more refined geometric and topological arguments  have been invoked. 

This can be seen in  the circle of ideas that originated in   \cite{JH} and \cite{HC},  starting from  simple covering spaces arguments (as in  \cite{P}),  and leading to:  degree theory  and the Palais-Smale condition 
in \cite{NX2};  foliations and intersection numbers in \cite {B1};  the Hopf fibration and the distortion theorems from the theory of univalent functions in \cite {NX1}; dynamics and horocycles in 
geometries of negative curvature in \cite{CX}, \cite{LX}, \cite{MX} and \cite{X4};  conformal geometry, partial differential equations  and  the Poincar\'e-Hopf theorem in the present paper. 
Needless to say,  
other authors have also made contributions from a topological standpoint 
to the global invertibility  program (see, for instance, \cite{R0}, \cite{R}).

The best known open problem in the area  of global invertibility is,  undoubtedly, the Jacobian conjecture ((JC), for short):
\vskip10pt 
\noindent \bf Conjecture (JC). \rm Every polynomial local biholomorphism $\mathbb C^n \to \mathbb C^n$ is invertible.
\vskip10pt
Hundreds of papers have been written about (JC), mostly couched in algebraic arguments (general references are \cite{BCW}, \cite{E}). 
On the other hand one should be cognizant of the fact that, despite casting  a long shadow on the subject,   (JC) is by no means the only outstanding  problem in the study  of global injectivity  and related issues.

 In order to dispel any misconception in this regard,  and  for the benefit of the reader,    
in  section 11  we compile a list of  roughly  fifteen research  problems,   from different areas of mathematics,  where the issue of   global injectivity  is at the forefront.

Returning  to the description  of the present work, we begin by refining the  global  injectivity  question:

\vskip10pt
\noindent \bf The singleton fiber  problem. \rm  For an individual point $q$ in the image of a local diffeomorphism $F:M^n \to N^n$ between non-compact  manifolds,  under what  extra conditions, involving only objects naturally associated to the point $q$ itself,    can one conclude that   the   fiber 
$F^{-1}(q)$  consists of a single point?
\vskip10pt

In order to take into account  the fact that the fibers of  a  local  diffeomorphism $F$  may have different sizes, as $F$ need not be a cover map,  (for instance, this happens in Pinchuk's famous counterexample to the  real Jacobian conjecture \cite {Pi}), ideally  any  additional  hypotheses required  to guarantee $\#F^{-1}(q)=1$   should   be sensitive to a possible change in the  cardinality  of nearby fibers,  and thus they should  involve only  objects   attached to $q$ itself. This is the rationale for including this condition in the above formulation of the singleton fiber  problem. The results in this paper reflect this point of view.

\vskip10pt

In the next section the reader  will find a short discussion of  the so-called reduction theorems in the Jacobian conjecture (\cite{BCW}, \cite{D}, \cite{Ja}),  which will offer further motivation  for the study of the singleton fiber  problem.

Our  aim  in this work is to  report on a    new connection between conformal geometry, topology,   elliptic partial differential equations,    and the singleton fiber problem  for a local diffeomorphism  $F:N^n\to \mathbb R^n$, $n\geq 3$. In the process, we will extend and give  a different perspective on some existing  theorems  on global  invertibility.

To  convey  the flavor  of our results,     we state  below a   special case of  an abstract  geometric-topological   criterion for a point to be covered only once (Theorem  \ref{geral}), that  also yields necessary and sufficient conditions for invertibility in the Jacobian conjecture (see Theorem \ref{Jac}   as well as   Theorem \ref{newJac}).

\begin{thm}  \label{apetizer} Let $F:\mathbb R^n \to \mathbb R^n$ be a  local diffeomorphism,  $n\geq 3$,  and $q\in F(\mathbb R^n)$.  If  the pre-image  of every plane containing $q$ is a Riemannian submanifold of Euclidean space that is conformal to $\mathbb R^2$, a once punctured sphere, then    $q$ is assumed exactly  once by  $F$.   
If, more generally, the pre-image of  every such plane is   conformal  to a finitely punctured sphere,  the number of punctures being allowed to vary with the plane, then  $q$ is assumed at  most twice
 by $F$. 
 \end{thm}

\vskip10pt
\begin{rem}  \rm The hypothesis that $n$ should be at least $3$ is essential. In fact, every non-injective local diffeomorphism $\mathbb R^2\to \mathbb R^2$  constitutes a counterexample in two dimensions.
We observe also that if $F:\mathbb R^n \to \mathbb R^n$ is  a  local diffeomorphism,  $n\geq 3$,  $q\in F(\mathbb R^n)$, and the pre-image of every plane $\pi\ni q$ is homeomorphic to $\mathbb R^2\setminus D_{\pi}$  where $D_{\pi}$ is discrete,  but $\#D_{\alpha}=\infty$ for at least one  plane $\alpha\ni q$, then the fiber $F^{-1}(q)$  may actually be infinite \cite{B}.
\end{rem}
\vskip10pt

The proof of Theorem \ref{apetizer}  is based on a geometric construction involving the Poincar\'e-Hopf theorem  on the relation between zeros of vector fields and the Euler  characteristic,  the  classical 
B{\^o}cher theorem on the structure of isolated singularities of positive harmonic functions in the plane,  condenser-like Dirichlet problems on Riemann surfaces, and the general theory of linear elliptic differential equations. 

The  underlying ideas are  rather conceptual, but the proof that certain sections of vector bundles are continuous require considerable work, as they involve using elliptic estimates on patches of manifolds, together with  compactness arguments to extract convergent subsequences of subsequences, etc. These  technical matters  will be discussed in section $4$ and beyond.

Since the ideas in this  work come  from several areas (complex analysis, algebraic and differential geometry,  topology and partial differential equations),  we  
provide  as much context and  supply  as many details as  reasonably possible.  In the same spirit, we hope that the broad list  of problems compiled in section 11 can  be useful to a reader interested in  other  aspects of global invertibility.

The author would like to  express his heartfelt gratitude to   his 
coauthor and now colleague,  
Scott Nollet,    for countless  enlightening conversations on the many facets of the theory  of global invertibility. Thanks are also due to E. Cabral Balreira for explaining to us the Gale-Nikaido conjecture (see \cite{GN} and section 11 of the present paper), and to Luis Fernando  Mello for his sustained interest in this work,  as well as for  explaining in detail   the Pinchuk maps \cite{MX}, \cite{Pi}.

This work  is affectionately dedicated to  my grandchildren, Lucca and Luna, ages 6 and 4. Hopefully  they will enjoy the mildly  anthropomorphic picture  in section 7.

\section {Vestigial forms of injectivity} 

In this section we discuss several examples to motivate the study of the problem of deciding when an individual fiber of a local diffeomorphism is  a singleton, paving the way for the results  to be  proved in later sections.

In Example 1 we recall the  reduction theorems from  the algebraic study of the Jacobian conjecture (JC).   This will be used to further motivate the study of the singleton fiber problem that was stated in the Introduction. In Example 2 we recall a result from \cite{NX1}, to the effect that a map as in (JC) is invertible if and only if the pre-images of complex lines are connected rational curves.

Examples 3 and  4  are  a preamble for Theorem \ref{apetizer}. In Example 5 we point out that any map as in (JC) is  ``coarsely  injective", in the sense that the pre-images of any \it real hyperplane \rm  is always connected (on the other end of the dimension spectrum, actual injectivity  is equivalent  to the condition that the pre-image of any \it point \rm  is connected). In Example 6 we discuss a simple conceptual application of differential geometry to the singleton fiber problem in the complex-analytic context.

\subsection{Maps with unipotent Jacobians}
\begin {exm} \label {example 1}  \rm 
Consider a polynomial  map $F=I+H:\mathbb R^n \to \mathbb R^n$, $n\geq 2$ arbitrary, where $H$ is homogeneous of degree $3$ and  $DH(x)$  is a nilpotent operator for every  $x$.
In particular,  the only eigenvalue of $DF(x)=I+DH(x)$ is $1$,  and so $F$ is a local diffeomorphism. 
Differentiating $H(tx)=t^3H(x)$ and setting $t=1$ one obtains the Euler relation $DH(x)x=3H(x)$, and from it   the functional equation
$DF(x)x=3F(x)-2x$.
It follows  from the latter that $0$ is covered exactly  once by $F$, that is $F^{-1}(0)=\{0\}$. Indeed,  if $x\neq 0$ and $F(x)=0$,  one has $DF(x)x=-2x$, contradicting the fact  that  $\text{Spec}\;DF(x)=\{1\}$. 

At first glance  the above  example seems  contrived, but  this is hardly the case. Rather than being  special artifacts, these  maps  actually incorporate    the general case of  the  above mentioned  Jacobian conjecture (JC) in algebraic geometry, thanks to the so-called reduction theorems of Dru\.{z}kowski \cite{D} and Jag\v{z}ev \cite{Ja} (see also \cite{BCW} and  \cite{E} for related developments).

The most  basic version of the  reduction theorems states that in order to prove (JC)   it suffices to establish injectivity for  all polynomial local biholomorphisms $G:\mathbb C^n \to \mathbb C^n$ of the form $G=I+K$, where $n$ is arbitrary,   $K$ is homogeneous of degree $3$,  and $DK(z)$ is nilpotent for every $z\in \mathbb C^n$.  We shall refer to such a 
$G$ as a Dru\.{z}kowski-Jag\v{z}ev map. In fact, in \cite{D}  the map $K$ is further refined to be  cubic linear.

Since the ground field is algebraically closed,  nilpotency is actually a consequence of homogeneity and local invertibility.  Conceivably, this (apparently new) observation can  be used to simplify the proofs of the reduction theorems.

To see this, and arguing more generally, let $G=I+K:\mathbb C^n \to \mathbb C^n$ be  a local biholomorphism,  where $K$ is homogeneous of degree $k>1$. Assume the existence of 
$z_0\in \mathbb C^n$, $\lambda \in \mathbb C \setminus \{0\}$, and $v\in \mathbb C^n \setminus \{0\}$ such that $DK(z_0)v=\lambda v$. If $\sigma\in \mathbb C$,  
\begin{eqnarray*} DG(\sigma z_0)v=v+DK(\sigma z_0)v=v+\sigma^{k-1}DK(z_0)v=(1+\sigma^{k-1}\lambda)v.\end{eqnarray*}
Since $\lambda\neq 0$, one can choose $\sigma$ so that $1+\sigma^{k-1}\lambda=0$. Hence  $DG(\sigma z_0)v=0$, contradicting the fact that $G$ is everywhere a local biholomorphism.  Hence, for every $z\in \mathbb C^n$, every eigenvalue $\lambda$ of $DK(z)$ must be zero. Thus,  homogeneity of $K$ implies that $DK(z)$ is nilpotent, as claimed in the previous paragraph.

Let now $G:\mathbb C^n \to \mathbb C^n$ be an arbitrary  holomorphic map, not necessarily polynomial,  $z_0\in \mathbb C^n$ and $\widehat {G}: \mathbb R^{2n}\to \mathbb R^{2n}$ its realification, where $\mathbb C^n$ is identified with $\mathbb R^{2n}$ in the usual way. The Jacobian determinants of these maps are related by the well-known formula,  
 a consequence of the Cauchy-Riemann equations:
\begin{eqnarray} \label {dois det} \text{det} D\widehat{G}(z_0) = |\text{det}  DG(z_0|^2. \end{eqnarray}

\noindent Applying (\ref{dois det}) to $G-\lambda I_{\mathbb C^n}$, first with $\lambda\in \mathbb R$, one has
\begin{eqnarray*}  \text {det}  (D\widehat{G}(z_0) -\lambda I_{\mathbb R^{2n}})=   \text{det} (D\widehat{G-\lambda I_{\mathbb C^{n}}  })(z_0) =
|\text{det} (DG(z_0)-\lambda I_{\mathbb C^n})|^2.
\end{eqnarray*}
From  
\begin{eqnarray*} |\text{det} (DG(z_0)-\lambda I_{\mathbb C^n})|^2= [\text{det} (DG(z_0)-\lambda I_{\mathbb C^n})]\overline{[\text{det} (DG(z_0)-\lambda I_{\mathbb C^n})]}  
\end{eqnarray*}
and
\begin{eqnarray*} \overline{\text{det} (DG(z_0)-\lambda I_{\mathbb C^n})} =\text{det} (\overline {DG(z_0)}-\lambda I_{\mathbb C^n}) 
\end{eqnarray*}
one has 
\begin{eqnarray} \label{real and complex}\text {det}  (D\widehat{G}(z_0) -\lambda I_{\mathbb R^{2n}})=[\text{det} (DG(z_0)-\lambda I_{\mathbb C^n})] 
[\text{det} (\overline {DG(z_0)}-\lambda I_{\mathbb C^n}) ], \;\;\lambda \in \mathbb R.
\end{eqnarray}
For $z_0$ fixed, (\ref{real and complex}) is an equality  between polynomials in $\lambda\in \mathbb R$, and so the identity persists  for  $\lambda \in \mathbb C$ as well. 

Observing that, for $\lambda \in \mathbb C$,  
\begin{eqnarray*} \text{det} (\overline {DG(z_0)}-\lambda I_{\mathbb C^n})= \overline { [\text{det} (DG(z_0)-\overline{\lambda} I_{\mathbb C^n})]},
\end{eqnarray*}
one can rewrite the analogue of (\ref{real and complex}) for $\lambda \in \mathbb C$ as
\begin{eqnarray}\label {de novo} \text {det}  (D\widehat{G}(z_0) -\lambda I_{\mathbb R^{2n}})=[\text{det} (DG(z_0)-\lambda I_{\mathbb C^n})] 
\overline { [\text{det} (DG(z_0)-\overline{\lambda} I_{\mathbb C^n})]}, \;\; \lambda \in \mathbb C.
\end{eqnarray}
As an immediate consequence of (\ref{de novo}), one has 
\begin{eqnarray} \label{conjugate}\text{Spec}[D\widehat{G}(z_0)]= \text{Spec} [DG(z_0)] \cup \overline{ \text{Spec} [DG(z_0)]}. \end{eqnarray}

If we now take $G=I+K$ to be a Dru\.{z}kowski-Jag\v{z}ev map in the discussion above, so that $DK(z)$ is nilpotent, it follows from (\ref{conjugate}) applied to $K$ that in 
$D\widehat G(z_0)= I +D\widehat K(z_0)$ the operator $D\widehat K(z_0)$ is nilpotent as well. 

In particular, the discussion in the first paragraph of this section, which was valid for real  maps,    applies when one takes $F=I+H=I+\widehat K: \mathbb R^{2n} \to \mathbb R^{2n}$. Consequently, $0$ is covered only once by $\widehat G$, the same being true of $G$ of course. 

After one uses the fact that a polynomial local biholomorphism is a cover map away from a certain algebraic hypersurface,   one has the following (known) invertibility criterion in (JC):

 \vskip10pt
 \noindent $\bf (\dagger)\rm $  \rm  (JC)  holds if and only if  one can show,  for any given   Dru\.{z}kowski-Jag\v{z}ev map $G=I+K:\mathbb C^n \to \mathbb C^n$,   that there exists some  neighborhood $U$ of $0$ in $\mathbb R^{2n}$  for which  $\underline{every}$   $q\in U$  is covered exactly once by the induced $\underline{real}$ map 
$\widehat G=I +\widehat K:\mathbb R^{2n}\to \mathbb R^{2n}$. \rm
\vskip10pt

The  simple-minded  argument using Euler's relation  from the beginning of this section  shows that  $q=0$ is covered only once by $\widehat G$,  but says nothing of the sort  if  $q\approx 0$.

In the context of the Jacobian conjecture injectivity implies surjectivity, and so $\bf (\dagger)\rm $   lends extra support  to  the need to investigate, even in the general case of local diffeomorphisms -  as opposed to just local biholomorphisms -, the sensitive question of  how the cardinality varies when one considers nearby  fibers. This is, in essence, the idea behind the singleton fiber  problem. 
\end{exm}

\vskip10pt

\subsection{The gradation   $(\diamondsuit_d)_{d=0}^{n-1}$  for coarse injectivity}

In order to provide context for the theorems in this paper, let us first consider the injectivity question from a  na\" \i ve topological standpoint.

It is  nearly tautological   that a point $q$ in the image of a local diffeomorphism $F: N^n \to \mathbb R^n$ is covered only once if (and only if) $F^{-1}(q)$ is connected. One is then led to the following   question, where $d$ is a non-negative integer, $0\leq d\leq n-1$:
\vskip10pt
\noindent   \bf Definition/Question (Coarse injectivity). \rm  If   the pre-image under a local diffeomorphism $F:N^n \to \mathbb R^n$ of all $d$-dimensional affine subspaces through $q$  are  connected, then one says that $(F, q)$ satisfies property $(\bf \diamondsuit_d)$\rm.  As observed above, property 
$(\bf \diamondsuit_0)$ \rm is equivalent to $\#F^{-1}(q)=1$. 
If  $d>0$, does  $(\bf \diamondsuit_d)$  imply $\#F^{-1}(q)=1$? \rm This question will be taken up in the next subsections. The answers  that we will obtain along the way will  lead us to the new theorems proven in this paper.
\vskip10pt

\begin{exm} \label {example 2} \rm 
The main impetus   for  exploring the idea of injectivity as a connectedness  phenomenon  stems from    the Jacobian Conjecture, since there are   powerful  results   in the literature about the connectedness of  algebraic varieties (Bertini theorems). 
An example in the algebraic setting  of the link between connectedness  and injectivity, albeit not directly motivated by  \bf $(\bf \diamondsuit_d)$\rm,  is the following injectivity criterion in the Jacobian conjecture, first established in \cite{NX1} as a consequence of a more general theorem
(see also Theorem \ref{Jac} in the present   paper for a new  proof): 
\vskip10pt
\noindent $\bf (\dagger \dagger)\rm $\rm A  polynomial local biholomorphism $F:\mathbb C^n \to \mathbb C^n$, $n>1$,  is invertible if and only if for the general  point $q$ in $F(\mathbb C^n)$, the pre-image under $F$ of every complex line containing $q$ is homeomorphic to an open domain  of $\mathbb R^2$ (i.e. the pre-image is a connected rational curve, in the language of algebraic geometry). \rm
\vskip10pt   
More generally, the main result in \cite{NX1} asserts the injectivity of  any    local biholomorphism $V^n \to \mathbb C^n$, $n>1$, for which the pre-images of  all complex lines  are biholomorphic to $\mathbb C\mathbb P^1$  punctured finitely many times, the number of punctures being allowed to vary with the complex line.
\end{exm}

\subsection{$(\diamondsuit_1)$   $\Longrightarrow (\diamondsuit_0)$ \rm}
\begin{exm}\label {example 3}  \rm 
\vskip10pt

As observed before,  \bf $(\bf \diamondsuit_0)$ \rm is equivalent to $q$ being covered only once.  Elementary arguments show that 
$(\diamondsuit_1)$   $\Longrightarrow (\diamondsuit_0)$ in the following form:  $\# F^{-1}(q)=1$   provided that 
$\underline{\rm some}$   $\underline {\rm line}$    through $q$  has a connected pre-image.  
To see this, let $l$ be a line containing $q$ for which $F^{-1}(l)$ is connected. By the classification theorem for connected $1$-manifolds, $F^{-1}(l)$ is diffeomorphic either to $\mathbb R$ or to a circle $S^1$. The restriction map 
$F^{-1}(l) \stackrel {F}\to l$ is a local diffeomorphism,  so  its image is an open set. If $F^{-1}(l)\approx S^1$, the image of the restriction map would be compact  (hence closed) and also open. 
By connectedness,  the image  would  be  the entire line $l$ and yet compact, a contradiction. It follows that  $F^{-1}(l)\approx \mathbb R$. Since a local diffeomorphism of $\mathbb R$ into itself is  strictly monotonic, hence injective,  $F^{-1}(l) \stackrel {F}\to l$  must also be  injective, and  therefore  $q$ is covered exactly once by $F$.  This establishes 
 \bf $(\bf \diamondsuit_1)$ \rm. 

Dynamically,  \bf $(\bf \diamondsuit_1)$ \rm can be understood as follows. Consider the two rays in $l$ with initial point $q$, and assume $a, b\in F^{-1}(l)$ are distinct  points mapped into $q$. Start lifting through $a$ the ray in $l$ that corresponds to moving along $F^{-1}(l)$ in the direction of $b$. If the ray cannot be lifted in its entirety then the lift rushed to infinity in finite time, in which case it passes through $b$. Evidently,  the lift also  contains $b$ if the ray can be lifted in its entirety. Hence, regardless, the maximal lift  through $a$ of the appropriate ray starting at  $q$ must contain 
$b$. Looking at the images, since $F(a)=F(b)=q$,  this means that at first the   image of the lift moved away from $q$,  along the line $l$,  but then it ``switched directions" and came back  to $q$. Hence,   the differential of $F$ must have become singular somewhere in $F^{-1}(l)$, contradicting the fact that $F$ is everywhere a local diffeomorphism.
\end{exm}

As the previous argument relies on  monotonicity,   one might expect the answer to  whether \bf $(\bf \diamondsuit_2)$ \rm  implies \bf $(\bf \diamondsuit_0)$ \rm to be  negative. This is indeed the case,  as shown by  the example below.
\vskip10pt
\subsection {Generic unknots and  $( \diamondsuit_2)$}
\begin{exm} \label{example 4}  

\rm We  exhibit  a local diffeomorphism $F:N^3\to \mathbb R^3$  for which  the fiber of some  $q\in F(N^3)$  has infinitely many  points, and yet  the pre-image of 
 every plane  passing through $q$ is a connected surface.

Start with   the  specific unknot $\Gamma\subset \mathbb R^3$  depicted in   Fig.1. Positioning $\Gamma$  appropriately, vis-\`a-vis the origin $0$,  the  following conditions  hold  for any two-dimensional vector subspace $V\subset \mathbb R^3$:  $0\notin  \Gamma$, $V\cap \Gamma$ is  non-empty,   finite, and there is at least one  point in 
$\Gamma \cap V$ where the intersection is transversal.

\vskip 5pt
\vspace {0.1cm}
\begin{figure}[!h]
\vspace{0.1cm}
\begin{center}
\includegraphics[scale=0.25]{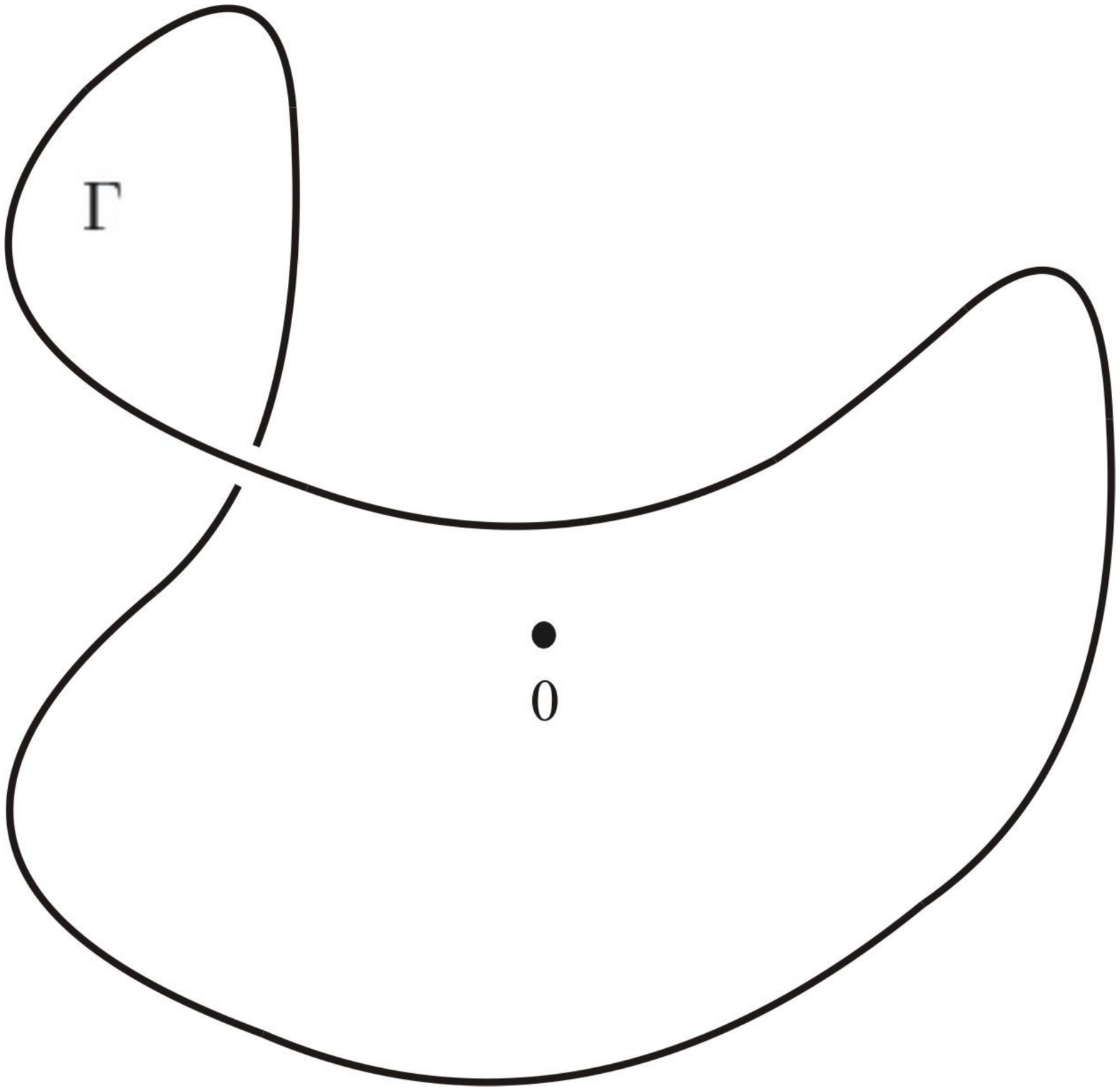}
\vspace{-0.4cm}
\vskip10pt
\caption{\footnotesize {For every plane $V$ through $0\notin \Gamma$,    $V \pitchfork_p \Gamma$ for at least one $p$ in the finite set $V\cap \Gamma$.}}
\label{Fig:02}
\end{center}
\vspace{-0.4cm}
\end{figure}

\noindent Let  $F:N \to \mathbb R^3 -\Gamma\subset \mathbb R^3$  be the universal cover map. As $\pi_1(\mathbb R^3-\Gamma, 0)= \mathbb Z$,   $F^{-1}(0)$ is infinite. 
The reader is asked to visualize that, given any   two-dimensional vector subspace $V$ of $\mathbb R^3$,    $A_V \stackrel{def}=V- (V \cap \Gamma)$ is connected.
 Also, a   ``thin" loop in $A_V$ that is  based  at $0$ and encircles only once some  point where the intersection $V\cap \Gamma$ is transversal  represents  a generator of  $\pi_1(\mathbb R^3-\Gamma, 0)$. 
It follows that the inclusion
$i:A_V \to \mathbb R^3 -\Gamma$ induces  an epimorphism
$ i_{*}: \pi_1(A_V, 0)\to \pi_1(\mathbb R^3-\Gamma, 0)$,
and so     $F^{-1}(V)=F^{-1}(A_V)$ is connected  \cite[p.179]{M}. Hence  $(\bf \diamondsuit_2)$ does not imply $(\bf \diamondsuit_0)$  (see also \cite{B} and problem xiv), section 11).

In section 4 we introduce a condition stronger than the topological condition $(\bf \diamondsuit_2)$, of  a conformal nature,   whose validity implies that $F^{-1}(q)$ is a singleton.
\end{exm}

\subsection{$(\diamondsuit_{\text{2n-1}})$ holds in  (JC)}

\begin{exm} \label{example 5} \rm 
  
Since  \bf $(\bf \diamondsuit_0)$ \rm and  \bf $(\bf \diamondsuit_1)$ \rm \rm both lead to  the fiber being  a singleton, but \bf$(\bf\diamondsuit_2)$ \rm does not, one can  think of  the property that the  pre-images of all $d$-dimensional affine subspaces are connected, for some $d$ in the range   $2\leq d \leq n-1$,      as a  coarse  form of injectivity.  

It is  amusing   that  this vestigial form of injectivity  always holds in the Jacobian conjecture setting,   for   the top value  $d=2n-1$. Of course, (JC) claims that  $(\bf \diamondsuit_d) $ already holds  at the level $d=0$.    

\vskip10pt
\begin{thm} \label{Bertini}  If $\widehat F:\mathbb R^{2n}\to \mathbb R^{2n}$ is the realification of a polynomial local biholomorphism 
$F:\mathbb C^n \to \mathbb C^n$, then $(\widehat {F})^{-1}(H^{2n-1})$ is connected for every real hyperplane  $H^{2n-1} \subset  \mathbb R^{2n}$. \end{thm}

This is a previously unpublished result  by S. Nollet and F. Xavier. Below, we provide the arguments when $n=2$. 

Write $z_1, z_2$ for the coordinates in the domain of $F=(F_1, F_2):\mathbb C^2\to \mathbb C^2$, and $w_1=u_1+iv_1$, $w_2=u_2+iv_2$  for the coordinates in the co-domain. The natural identification between $\mathbb C^2$  and $\mathbb R^4$ is $(w_1, w_2) \to (u_1, v_1, u_2, v_2)$.  We are supposed to show that $F^{-1}(H^3)$ is connected for every real hyperplane $H^3$. Choosing coordinates appropriately, it suffices to take $H^3$ to be the hyperplane $u_1=0$. 

Consider the projection  
$\Lambda :H^3\to \mathbb R$,  $\Lambda(0, t,  v_1, v_2) =t$. Under the identification,  $H^3$ is foliated by the complex lines $V_t=\{(w_1, w_2) \;|\; w_1=it\}$, 
$t \in \mathbb R$. Observe  that $F^{-1}(V_{t})=(F_1)^{-1}(it)$ is non-empty for every $t$ because the polynomial $F_1:\mathbb C^n \to \mathbb C$ is non-constant, hence surjective by the fundamental theorem of algebra. 

A complex polynomial $p(z_1, z_2)$ is said to be \it primitive \rm  if the fibers $\{p=c\} \subset \mathbb C^2$ are connected,  with the possible exception of finitely many values of $c$.  It is a classical result that if $p$ is $\underline{\rm not}$ primitive  then  it can be decomposed as $p=q\circ r$, where $q:\mathbb C \to \mathbb C$ is a polynomial of degree at least two, and $r$ is a polynomial in two variables. We claim  that if $p$ has no critical points, then $p$ is primitive. 

In order to see this, we argue by   contradiction.  In the decomposition  $p=q\circ r$ the polynomial $r$ is necessarily non-constant, and so it assumes all complex numbers.  The gradients satisfy $\nabla p(z)=q'(r(z))\nabla r (z)$.  If $z$ is chosen so that $r(z)$ is a zero of $q'$ (which is possible  because 
$\text {deg} \; q'\geq 1$),  then $\nabla p(z)=0$, a contradiction. This establishes our claim. 

Observe that the polynomial  $F_1$ has no critical points (otherwise $\text {det} DF$ would vanish somewhere) and therefore, by the previous paragraph, $F_1$  is primitive.  
In particular,  there areat most finitely many values of $t\in \mathbb R$, say $t_1, \dots, t_k$, for which the (necessarily non-empty) set $F^{-1}(V_{t_j})=(F_1)^{-1}(it_j)$ is disconnected. 

We can now show that $F^{-1}(H^3)$ is connected.  Again arguing by contradiction, write 
$$F^{-1}(H^3)=\displaystyle \bigcup_{t\in \mathbb R} F^{-1}(V_t)=U_1\cup U_2,$$ 
where $U_1$, $U_2$ are non-empty, open, and disjoint.  
Set 
$J_k= \{t\in \mathbb R : F^{-1}(V_t)\cap U_k\neq \emptyset\}.$
As remarked before,  $F^{-1}(V_t)\neq \emptyset$ for every $t\in \mathbb R$,   so that $J_1\cup J_2=\mathbb R$. Since $F$ is locally invertible, each $J_k$ is a non-empty open set. From the connectedness of $\mathbb R$, $J_1\cap J_2$ is a non-empty open set.  
In particular, we can select
 $\overline t \in J_1\cap J_2\setminus\{t_1, \dots, t_ k\}$.  It is now clear that 
 $F^{-1}(V_{\overline t})\cap U_1$ and $F^{-1}(V_{\overline t})\cap U_2$ form a disconnection of $F^{-1}(V_{\overline t})$, a contradiction. This shows  that $F^{-1}(H^3)$ is connected, concluding the proof of Theorem \ref{Bertini}.

\end{exm} 

\section{A differential-geometric mechanism  for injectivity}
\noindent \begin{exm} \label{example 6}  \bf  \rm
The following example from \cite{NX1} is  a conceptual  application of topology and differential geometry to the singleton fiber problem,  in the context of complex analysis. It is based on the fact that the Hopf map does not admit continuous sections.

\vskip10pt
\noindent ($\bullet $)  \rm Let $F:\mathbb C^n \to \mathbb C^n$, $n\geq 2$, be a local biholomorphism and $q$ a point in the image of $F$. If the pre-image of every complex line $l$ containing $q$ is connected and simply-connected, then $q$ is assumed exactly once by $F$. 
\rm
\vskip10pt

To see this, assume that $F(a)=F(b)=q$, $a\neq b$.  We may assume $q=0$.  By hypothesis, the complex curve $F^{-1}(l)$ contains $a$ and $b$ for every such $l$.  Since $F^{-1}(l)$ is properly embedded (because $F$ is a local bihilomorphism), it follows that $F^{-1}(l)$ is a complete, connected, simply-connected  Riemannian  surface of non-positive curvature. 

By the Cartan-Hadamard theorem  alluded to in the Introduction, there is a unique geodesic in $F^{-1}(l)$ that connects $a$ to $b$. Let $v_l$ be the unit tangent vector at $a$ of this geodesic. The set of complex lines passing through $q$ is naturally identified with $\mathbb C \mathbb P^{n-1}$. We can then define a map 
\begin{eqnarray} \label{Hopf} s:\mathbb C \mathbb P^{n-1} \to S^{2n-1}, \;\; s(l)=\frac{DF(a)v_l}{|DF(a)v_l|}\in l\cap S^{2n-1}.\end{eqnarray}
Consider the Hopf map $\pi:S^{2n-1}\to \mathbb C \mathbb P^{n-1}$ that associates to a point in the unit sphere the unique complex line passing through  it and the origin.   It follows from (\ref{Hopf}) that  $s$ is a section of $\pi$. Since one can argue that $s$ is continuous,  invoking the continuous dependence of solutions of systems of ordinary differential equations upon the initial conditions as well as parameters,  we arrive at a contradiction. 

A considerable elaboration of the above arguments leads to a sharp  estimate on the cardinality of the fibers of certain local biholomorphisms.  The theorem below is a special case of the main result of \cite{CX}:

\begin{thm} \label {fiber estimate} Let $F:\mathbb C^2 \to \mathbb C^2$  be a local biholomorphism, $d$ a positive integer, $q$ a point in the image of $F$, and $\mathcal L$ the set of all complex lines passing through $q$. Assume that :
\vskip5pt
\noindent i) For every $l\in \mathcal L$ the complex curve $F^{-1}(l)$ has at most $d$ connected components.
\vskip3pt
\noindent ii)    There exists a finite set $\mathcal F \subset \mathcal L$ such that $F^{-1}(l)$ is  $1$-connected  for all 
 $l\in \mathcal L -\mathcal F$.
\vskip5pt
\noindent Then the fiber $F^{-1}(q)$ has at most $d$ points.
\end{thm}

That the theorem is sharp can be seen  by taking $q=0$ and $F(z, w)=(f(z), w)$,  where $f:\mathbb C \to \mathbb C$ is holomorphic, $f'$ is nowhere zero, and $\#f^{-1}(0)=d$ (see \cite{CX} for details).
\end{exm}

\section{\bf  A topological-conformal enhancement of  $(\bf\diamondsuit_2)$}

Notwithstanding Example \ref{example 4},  it is possible to strengthen   $(\bf\diamondsuit_2)$ so that the enhanced property implies $\#F^{-1}(q)=1$. Loosely speaking, one  needs to go  beyond mere  connectedness by requiring that    the   pre-images of all   planes  containing the point $q$  be   diffeomorphic   to $\mathbb R^2$ and  also ``conformally large".

It is  a classical result   that every orientable Riemannian surface $(M, g)$  can be given a Riemann surface structure, where  local holomorphic coordinates $z=x+iy$ can be introduced so that, relative to these so-called isothermal parameters the metric $g$  is locally conformally flat, meaning that it assumes the form
$g(x,y)=\lambda^2(x,y) (dx^2+dy^2)$
for some  smooth function $\lambda>0$ (\cite{C},   \cite[p.~19]{GMP},  \cite[Thm.~3.11.1] {J}).

 Unless otherwise stated, we shall agree that  in this paper  the conformal (or complex) structure associated  to any  orientable surface   $M$  embedded in some  euclidean space
$\mathbb R^k$   will be the one arising from  the Riemannian metric  on $M$  obtained by  the restriction of the standard inner product  of $\mathbb R^k$.
\vskip5pt
We write $q+\pi$ for the  plane containing the point $q$ that is parallel to the two-dimensional vector subspace $\pi$ of $\mathbb R^n$. As usual, $G_2(\mathbb R^n)$ stands for the Grassmannian of two-dimensional vector subspaces of $\mathbb R^n$.
The following result is   stated first in the language   of differential geometry. 
\begin{thm}\label{parabola}
Let   $F:\mathbb R^n \to \mathbb R^n$  be a local diffeomorphism, $n\geq 3$,    $q\in F(\mathbb R^n)$.  Assume the existence of  a   compact orientable surface $M^2$ embedded in $\mathbb R^n$  such that  $\chi(M^2)\neq 0$ and,  for every $\pi \in G_2(\mathbb R^n)$ that is parallel to some  tangent plane of $M^2$,    the surface $F^{-1}(q+\pi)\subset \mathbb R^n$   is conformally diffeomorphic    to $ \mathbb R^2$. Then   the fiber $F^{-1}(q)$ consists of a single point.
\end{thm}

In complex-analytic terms, the hypothesis of Theorem \ref{parabola}  stipulates   that   the (necessarily non-compact) one-dimensional complex manifold  $F^{-1}(q+\pi)$   is   connected, simply-connected,   and biholomorphic  to $\mathbb C$ rather than to an open disc $D$ (by the Koebe uniformization theorem,   a $1$-connected Riemann surface is biholomorphic   to
$\mathbb C \mathbb P^1$, $\mathbb C$, or $D$).
\vskip5pt

The next result yields  information about the  cardinality of the fiber $F^{-1}(q)$ when $F^{-1}(q+\pi)$  is allowed to be conformal to a  sphere $S^2$  punctured finitely many times (and not  just  once,  as in Theorem \ref{parabola}):
\begin{thm} \label {parabola dois}   Let   $F:\mathbb R^n \to \mathbb R^n$  be a local diffeomorphism, $n\geq 3$,    $q\in F(\mathbb R^n)$.  Assume the existence of  a   compact orientable surface $M^2$ embedded in $\mathbb R^n$  such that $\chi(M^2)\neq 0$ and,  for every $\pi \in G_2(\mathbb R^n)$ that is parallel to some tangent plane of $M^2$,  there exist points $p_1, \dots, p_{k_\pi} \in S^2$, $k_{\pi} \geq 1$,   for which  the surface $F^{-1}(q+\pi)\subset \mathbb R^n$   is   conformally diffeomorphic   to
$S^2-\{p_1, \dots, p_{k_\pi}\}$.
Then   the fiber $F^{-1}(q)$ has at most two points.
\end{thm}

Taking $M^2$ to be a compact orientable surface in $\mathbb R^n$ other than a torus in Theorems \ref{parabola} and \ref{parabola dois}, one has:
\vskip10pt
\begin{cor} \label{injectivity} Let $F:\mathbb R^n \to \mathbb R^n$ be a local diffeomorphism,  $n\geq 3$. If the pre-image under $F$ of every  affine plane  is a surface conformal to $\mathbb R^2$, then $F$ is injective.
\end{cor}

\begin{cor} \label{Corolario 2} Let   $F:\mathbb R^n \to \mathbb R^n$  be a local diffeomorphism, $n\geq 3$.   If the pre-image under $F$ of every affine plane    is  a surface conformal  to a finitely punctured sphere, 
the number of punctures being allowed to vary with the plane,  
then   every point in $\mathbb R^n$  is covered at most twice  by $F$.
\end{cor}
The two-dimensional analogue of Corollary \ref {injectivity}  fails  dramatically. Indeed,  $\underline {\rm every}$ non-injective local diffeomorphism
$F: \mathbb R^2 \to  \mathbb R^2$ provides a counterexample. Examples include the realification of the complex exponential and  the Pinchuk  counterexamples to the so-called real Jacobian conjecture \cite {Pi}.

A reasonable   conjecture  seems to be that  the estimate  $\#F^{-1}(q)\leq 2$ in Theorem \ref{parabola dois} can be improved to $\#F^{-1}(q)=1$,  a  result that  would of course be stronger than   the conclusion in Theorem \ref{parabola}.

The proof of Theorem \ref{parabola} is fairly long and technical, but we can  highlight  here some of  the main points.
The central idea is that if the fiber $F^{-1}(q)$ had at least two  distinct elements, then  one could construct a continuous nowhere zero
vector field  on $M^2$, thus contradicting the   Poincar\'e-Hopf theorem since $\chi(M^2)\neq 0$.

The construction of  this vector field  explores  the special  nature   of   the
isolated singularities of    positive harmonic functions   in the plane.
To establish   continuity  -- the more involved  part of the  proof --,   one  uses B{\^o}cher's theorem and   the    standard theory of  linear  elliptic  partial differential equations,  
after some technical difficulties have been overcome. 

In Section 7 we   give  a  fairly  conceptual  proof of Theorem  \ref{parabola dois}. It uses   the solutions of certain  Dirichlet problems on the Riemann surfaces $F^{-1}(q+\pi)$,  in a manner  reminiscent of the
notion of  \it condensers \rm in classical potential theory,  in order  to create, under the assumption that the fiber has at least three elements, a continuous non-vanishing vector field on $M^2$. This contradiction shows that the fiber  has at most two elements.

These results are actually special cases of a general abstract mechanism for global injectivity, codified in Theorem \ref{geral}, that  may yet be useful in other settings. With the aid of  the algebro-geometric result that the generic fiber of a polynomial local biholomorphism $F:\mathbb C^n \to \mathbb C^n$
cannot have cardinality two, Theorem \ref{geral}  yields  another  proof of the    necessary and sufficient condition for invertibility in (JC)  that  was first established in  \cite{NX1} (see Example \ref{example 2} in  section 2 and   Theorem \ref{Jac}).  

Theorem \ref{newJac} is a new result,  a necessary and sufficient condition for invertibility in the Jacobian conjecture  when  the polynomial map has $\underline{\rm real}$  coefficients. It stands as   a natural companion to Theorem \ref{Jac}.  Besides the condenser construction from section 7,  the  proof of Theorem \ref{newJac} uses ``symmetry" arguments and  the fact that the Euler  number of the tautological line bundle  is non-zero mod\;(2).

\section{Using B{\^o}cher's theorem    to   
create  vector fields 
}

In this section we begin the proof of Theorem \ref{parabola}, which is couched on two  closely related classical results about isolated  singularities of positive harmonic functions.  

Lemma \ref{segundo lema} will be used in this section in the course of a geometric construction aimed at establishing the  \it existence \rm  of certain vector fields on $M^2$. Lemma \ref{B} - actually a theorem of  B{\^o}cher -,   will play a central  role in the 
\it continuity \rm  proof of  the said vector field. This is  a lengthy argument, presented in section 6,     that  also uses  the general theory of linear elliptic equations.
\vskip5pt 

\noindent In what follows, we set   $B=\{z\in \mathbb C : |z|<1\}$.

\vskip10pt

\begin {lem} \label {segundo lema} \rm (\text  {\cite{ABR}, p. 52}) \it  Let  $u$  be  positive  and harmonic on   $B\setminus \{0\}$,  $u(z)\to 0$ as $|z|\to  1$. Then there exists a constant $c>0$ such that
$u(z)=-c\log |z|$, $z\in B\setminus \{0\}$.
\end{lem}

\begin{lem}  \label{B}   \rm (\text  {\cite{ABR}, p. 50})   \it Let  $u$  be positive and   harmonic  on $B\setminus \{0\}$. Then there is a harmonic function $\nu$  on $B$ and a constant $c\geq 0$, with 
$u(z)=\nu(z)-c\log |z|,  z\in B\setminus\{0\}.$
\end{lem}

Let   $F$ be as in the statement of Theorem \ref{parabola}. Replacing $F$ by $F-q$ we may, and will,  assume that  $q=0$.
Suppose, by contradiction, that $F^{-1}(0)$ contains at least two  distinct points,  say $a$ and $b$.
Consider an  open neighborhood $U$ of $a$  and  an  Euclidean ball $W$ centered at $0$   such  that   $b\notin \overline U$ and $U\xrightarrow[]{F}W$  is a diffeomorphism.

Let $\pi\subset \mathbb R^n$ be a plane containing $q=0$ that is parallel to some tangent plane of $M^2$. We shall denote by $\mathcal A$ the totality of such $\pi$, so that $\mathcal A$ may  be regarded as a subset  of the Grassmannian   $G_2(\mathbb R^n)$ of two-dimensional  vector subspaces of $\mathbb R^n$.  
Observe that    $a,b\in F^{-1}(\pi)$ for all  $\pi\in \mathcal A$.  
We will abuse the notation somewhat and write interchangeably $\pi\in G_2(\mathbb R^n)$ or  $\pi\subset \mathbb R^n$.   
Set 
\begin{eqnarray} \label{U and T} U_{\pi}=U\cap F^{-1}(\pi),   \;\;T_{\pi}=\partial (U\cap F^{-1}(\pi)). \end{eqnarray}

Denote by $\Delta_{\pi}$ the Riemannian  Laplacian on the simply-connected surface $F^{-1}(\pi)$,  associated to the metric $g_{\pi}$ on $F^{-1}(\pi)$ obtained by the restriction of the  Euclidean metric $g$ of $\mathbb R^n$.

In  local conformal coordinates $z=x+iy$ on $F^{-1}(\pi)$,    $g_{\pi}$ and $\Delta_{\pi}$ are given by
\begin{eqnarray*} g_{\pi}=\lambda^2 |dz|^2, \;\;   \Delta_{\pi}=4\lambda^{-2}\partial_z\partial_{\overline z},\end{eqnarray*}
where $\lambda$ is a positive smooth function,  and
\begin{eqnarray*} \partial_z=\frac{1}{2}(\partial _x -i\partial_y),  \;\;\; \partial _{\overline z}= \frac{1}{2}(\partial _x +i\partial_y). \end{eqnarray*}

Hence, a function on $F^{-1}(\pi)$ is $\Delta_{\pi}$-harmonic (i.e. harmonic in the sense of the Riemannian metric $g_{\pi}$)  if and only if  it is harmonic in the  sense of the Riemann surface structure of $F^{-1}(\pi)$ induced by the isothermal parameters $z$ above.

\begin{lem} \label {terceiro  lema} For each $\pi$ as before there exists a unique function $u_{\pi}: F^{-1}(\pi)\setminus \overline{U}_{\pi}\to [0, \infty)$
satisfying the following properties:
\vskip3pt
\noindent i) $u_{\pi}$ is  $\Delta_{\pi}$-harmonic.
\vskip3pt
\noindent ii)   $\lim u_{\pi}(p)=\infty$ uniformly as $|p|\to \infty$ in $\mathbb R^n$,  $p\in F^{-1}(\pi) \setminus \overline {U_{\pi}}$.
\vskip3pt
\noindent iii)  $\lim u_{\pi}(p)=0$  as  $p \to T_{\pi}$,  $p\in F^{-1}(\pi) \setminus \overline {U_{\pi}}$.
\vskip3pt
\noindent iv) $u_{\pi}(b)=1$.
\vskip3pt
\noindent Furthermore, $du_{\pi}(b)\neq 0$.
\end{lem}

The proof of Lemma \ref{terceiro lema}   proceeds as follows. Recall that we are working under the assumption that, for all $\pi \in \mathcal A$,  the Riemann surface $F^{-1}(\pi)$ is   conformal to the complex plane $\mathbb C$.   
Hence  $F^{-1}(\pi)\setminus  \overline U$ is conformal to a punctured neighborhood of infinity in the Riemann sphere which, in turn, is conformal to $B\setminus \{0\}$.

As the biholomorphic self-maps of  $B\setminus \{0\}$ are rotations about $0$,  it follows that  there is a \it unique \rm conformal map $h_{\pi}: F^{-1}(\pi)\setminus  \overline {U_{\pi}} \to B\setminus \{0\}$ such that  $h_{\pi}(b)\in (0,\infty)\cap B$.

Consider the function $v_{\pi}:B\setminus \{0\} \to (0,\infty)$,

\begin{eqnarray} \label{log}  v_{\pi}(z)=\frac{\log|z|}{\log (h_{\pi}(b))}, \;\; z\in B\setminus \{0\}. \end{eqnarray}

Observe that 
\begin{eqnarray} \label {double} |\lim_{p \to T}h_{\pi}(p)|=1, \;\;\;  \lim_{||p ||\to \infty}h_{\pi}(p)=\infty \end{eqnarray}

It is now clear that $u_{\pi} \stackrel {def}=v_{\pi}\circ h_{\pi}$ satisfies  properties i)-iv) above, settling the existence part of Lemma \ref{terceiro lema}. Furthermore, $du_{\pi}(b)\neq 0$ because $v_{\pi}$ has no critical points. 

In order to  establish uniqueness, assume that  $\tilde u_{\pi}$  is non-negative and satisfies i)-iv).
It follows from  (\ref{double}) and   iii)  that $\tilde v_{\pi}:= \tilde u_{\pi}\circ (h_{\pi})^{-1}:B\setminus\{0\} \to (0, \infty)$ satisfies the hypothesis of Lemma \ref{segundo lema}, and so
$ \tilde v_{\pi}(z)=-c\log |z|$
for some $c>0$ and all  $z\in B\setminus \{0\}$.
From (iv),
$$1= \tilde u_{\pi}(b)=\tilde v_{\pi}(h_{\pi}(b)) = -c\log |h_{\pi}(b)|.   $$
Solving for $-c$ and using (\ref{log}) one has  $v_{\pi}=\tilde v_{\pi}$, and so  $ u_{\pi}=\tilde u_{\pi}$. This concludes the proof of Lemma \ref{terceiro lema}.  

\qed
\vskip10pt
Continuing with the proof of Theorem \ref{parabola},  for $p\in M^2$ let $\pi_p =T_pM^2$ be the tangent \it space  \rm  of $M$ at $p$ (hence a two-dimensional vector space, naturally identified with an element of $G_2(\mathbb R^n)$).

Recall that the  Riemannian   gradient $\nabla  u_{\pi_p}(b)$, where $u_{\pi_p}$ is given by Lemma \ref{terceiro lema},  is the vector in the  tangent space
$T_b(F^{-1}(\pi_p))$ that represents, in the inner product sense,  the differential
$du_{\pi_p}(b):T_b(F^{-1}(\pi_p))\to \mathbb R$. 

\vskip10pt
\begin{lem}\label{quarto lema} The map $s:M^2\to \mathbb R^n$ given by
\begin{eqnarray} \label {def} s(p)=dF(b)\nabla u_{\pi_p}(b), \, \;p\in M^2, \end{eqnarray} 
defines   a non-vanishing  vector field on $M^2$.
\end{lem}
Indeed, for all $p\in M^2$ one has $\nabla  u_{\pi_p}(b) \in T_b(F^{-1}(T_pM^2))$, and since $b\in F^{-1}(T_pM^2)$, on account of $F(b)=q=0$,   one has
\begin{eqnarray*} \label{s} s(p)= dF(b)\nabla u_{(T_p M^2)}(b)\in dF(b)[T_b(F^{-1}(T_pM^2))]=T_{F(b)}T_pM^2=T_0T_pM^2 \approx T_pM^2,\end{eqnarray*}
so that $s$ is a section of the tangent bundle $TM^2$, i.e a vector field on $M^2$.  
Since $\nabla u_{(T_pM^2)}(b)\neq 0$ by Lemma \ref{terceiro lema}, and $dF(b)$ is non-singular ($F$ is a local diffeomorphism), it follows from (\ref{s})  that $s(p)\neq 0$ for all 
$p\in M^2$.

\qed
\vskip10pt

\noindent \bf Proof of Theorem \ref {parabola}. \rm
\vskip10pt

It will be shown in the next section, using a lengthy argument,  that the nowhere zero vector field on $M^2$ defined     in Lemma \ref{quarto lema}, and whose construction depends in an essential way on the assumption
$\#F^{-1}(q)=\#F^{-1}(0)\geq 2$,  is  continuous. But this contradicts the Poincar\'e-Hopf theorem \cite{OR} because $\chi(M^2)\neq 0$, and so $F^{-1}(0)$  must be   a singleton,  thus completing the proof of
Theorem \ref {parabola}.

\section{Elliptic estimates and the continuity argument}

Our goal in this section is to  finish the proof of Theorem \ref{parabola} by showing  that the   vector field $s$ given in Lemma \ref{quarto lema} is continuous.
The proof uses the general theory of linear elliptic partial differential equations, as well as the B\^ocher theorem (Lemma \ref{B}). 

Recall from the previous section that we are working under the assumption that $F(a)=F(b)=0$,  $a\neq b$. To keep the notation at a  manageable level, we continue setting
$ \pi_p=T_pM^2$. The vector field $s$ from Lemma \ref{quarto lema} is given by  
\begin{eqnarray}\label{new def}  s(p)=dF(b)\nabla u_{\pi_p}(b)\in T_{F(b)}T_pM^2=T_0T_pM^2\cong T_pM^2, \; p\in M^2. \end{eqnarray}
In order to  establish the continuity of $s$ at  $z\in M^2$, it suffices to show that for every sequence $(z_k)$ in $M^2$  that converges to $z$  there is a subsequence $z_{k_j}$ such that
\begin{eqnarray} \label {subsequence} s(z_{k_j})=dF(b)\nabla u_{\pi_{z_j}}(b)\to dF(b)\nabla u_{\pi_z}(b) = s(z).\end{eqnarray}

Loosely speaking, the desirable strategy  would be to first      pull-back  to $F^{-1}(\pi_z)$ the functions 
$u_{\pi_{z_k}}$ that are  given by Lemma \ref{terceiro lema}. Next, one would use
 compactness arguments  for families of positive harmonic  functions to  take $C^k$ limits over compact subsets of $F^{-1}(\pi_z)$  of the pull-backs of $u_{\pi_{z_k}}$ along suitable subsequences,   in order to obtain  a harmonic function $\hat u_{\pi_z}$ that  would  then be  shown to satisfy all hypotheses of Lemma \ref{terceiro lema}. 
 By the uniqueness part of the said lemma one would have    $\hat u_{\pi_z}=u_{\pi_z}$. In particular,  $\nabla \hat u_{\pi_z}(b)=\nabla u_{\pi_z}(b)$,   and (\ref{subsequence}) would follow.  
 The general  plan can be summarized schematically as
  $$ \text{ Compactness} + \text {Uniqueness} \Longrightarrow \text {Continuity}. $$
  
The main  difficulty  in carrying out  such an  approach is that there is no natural \it globally defined \rm   map $F^{-1}(\pi_z) \to F^{-1}(\pi_{z_k})$ that would allow us to take pull-backs. 

However,  as $z_k\to z$,  for any   fixed compact set
 $C\subset F^{-1}(\pi_z)$ and  sufficiently large index $k$,  there are natural injections $C\to F^{-1}(\pi_{z_k})$,   because  the surfaces $F^{-1}(\pi_{z_k})$ converge $C^r$-uniformly over compact sets of $\mathbb R^n$ to $F^{-1}(\pi_z)$. 
 
 This will allow us to pull back functions, over fixed compact sets $C=C_m$ from an exhaustion of  $F^{-1}(\pi_z)$ and, as we shall see,  with additional  arguments, working with a fixed compact set in an exhaustion at a time,  extracting convergent subsequences of subsequences, and so on,  the above outline can be implemented.

\vskip10pt
\noindent \it Notation. \rm From this point on, in order to simplify the writing,  we set 
$\pi_k= \pi_{z_k}, \;\; \pi=\pi_z$.
\vskip10pt

The continuity proof will make use of   results for linear elliptic partial differential equations that are familiar to the experts  (the  maximum principle, Harnack's  inequality,  interior and boundary  estimates, etc.). But in terms  of the presentation  the arguments are not entirely straightforward,  due the fact that the relevant harmonic functions are defined on different spaces, precluding  us from simply quoting these standard  theorems. 
 
We now begin the formal proof of (\ref{subsequence}).  Certain  portions of it, having to do with setting up the procedure to pull back functions,  are patterned after some arguments in \cite{NX1}. 

As remarked above, the properly embedded surfaces $F^{-1}(\pi_{k})$ converge    to $F^{-1}(\pi)$,  in the $C^r$ sense, for any $r\geq 1$ fixed, uniformly over   compact subsets of $\mathbb R^n$.  

Let $K=K_1$ be the first  set  in a countable  increasing exhaustion $\{K_m\}_{m=1}^{\infty}$ of
$\mathbb R^n$ ($=$ the domain of $F$) by closed balls, where the radius of $K_1$ is large enough so that $a, b$ lie in the same component $C_1$ of  $K_1\cap F^{-1}(\pi)$. Clearly, 
$a, b$ lie in the same component $C_m$ of $K_m\cap F^{-1}(\pi)$, $m\geq 1$.

One can cover  $K_1$  with finitely many Euclidean  balls, with a  sufficiently small radius,  that are mapped by $F$ diffeomorphically  onto their images. Take a   Lebesgue number  $5\delta$ for this cover. This  means  that any subset of diameter less than $5\delta$ is contained in an open set of the said covering. If $z_1, z_2\in K_1$ and  $B_{\delta}(z_1)\cap B_{\delta}(z_2) \neq \emptyset$,   then  the distance between $z_1$ and $z_2$ is at most $2\delta$ and so $B_{\delta}(z_1)\cup B_{\delta}(z_2)$ has diameter at most $4\delta$. It  follows that:
\vskip10pt
\noindent i)  If $B_{\delta}(z_1)\cap B_{\delta}(z_2) \neq \emptyset$ then the restriction $F|(\overline{B_{\delta}(z_1)}\cup \overline {B_{\delta}(z_2)})$ is injective.
\vskip10pt
Choose  $\tilde z_1, \dots, \tilde z_m \in C_1$ so  that the open balls $B_{\delta}(\tilde z_j)$, $1\leq j \leq m$,  cover $C_1$.
Let $S_n\in O(n)$  be a sequence of orthogonal transformations that converge to the identity map 
$I=I_{\mathbb R^n}$  and
satisfy $S_k(\pi)=\pi_k$.

As  $S_k\to I$ and $C_1$ is compact,   there exists an integer $N$ such that  for $k\geq N$,
\begin{eqnarray*} S_k(F(z))\in F(\overline{B_{\delta}(\tilde z_j)}), \end{eqnarray*}
for all $j\in \{1, \dots, m\}$  whenever  $z\in C_1\cap B_{\delta}(\tilde z_j)$.

Consider, for $k\geq N$,  the map
$\phi_{k}: C_1\to F^{-1}(\pi_k)\subset \mathbb R^n$, given by
\begin{eqnarray}\label{tilde} \phi_{k}(z)=[(F|_{\overline{B_{\delta}(\tilde z_j)}})^{-1}\circ S_k\circ F](z), \end{eqnarray}
where $\tilde z_j$ is such that $z\in C_1\cap B_{\delta}(\tilde z_j).$ 

One must of course show that the definition of $\phi_k$ is independent of the point $\tilde z_j$.
Suppose therefore that
\begin{eqnarray}\label{tilde tilde} [(F|_{\overline{B_{\delta}(\tilde z_j)}})^{-1}\circ S_k\circ F](z)=[(F|_{\overline{B_{\delta}(\tilde z_i)}})^{-1}\circ S_k\circ F](z), \end{eqnarray}
where $\tilde z_i$ is such that $z\in C_1\cap B_{\delta}(\tilde z_i)$,  and similarly for $\tilde z_j$.
Write $a_j$ and $a_i$ for the  left and right hand sides of  (\ref{tilde tilde}), so that $a_j\in \overline{B_{\delta}(\tilde z_j)}$, $a_i\in \overline{B_{\delta}(\tilde z_i)}$ and
$F(a_j)=S_k(F(z))=F(a_i).$

Since $a_i, a_j\in \overline{B_{\delta}(\tilde z_i)} \cup \overline {B_{\delta}(\tilde z_j)}$,  $z\in  B_{\delta}(\tilde z_i) \cap B_{\delta}(\tilde z_j)$ and $F(a_i)=F(a_j)$, it follows from i) that $a_i=a_j$, and so
$\phi_k$ given by (\ref{tilde}) is a well-defined  local diffeomorphism.

Replacing $S_k$  by $I_{\mathbb R^n}+(S_k-I_{\mathbb R^n})$ in (\ref{tilde}), we write
\begin{eqnarray} \label {tending to the inclusion}\phi_{k}(z)= z+[(F|_{B_{\delta}(\tilde z_j)})^{-1}\circ (S_k-I_{\mathbb R^n})\circ F](z) \in F^{-1}(\pi_k), 
\;\;z\in C_1.\end{eqnarray}

 The term in brackets in
 (\ref{tending to the inclusion}) goes to zero, together with any finite number of derivatives, uniformly over $C_1$, as $k\to \infty$.  In other words, by  taking a possibly larger $N$ one can view $\phi_k$ as an arbitrarily small perturbation of the identity.

 Since $C_1$ is compact, it follows from (\ref{tending to the inclusion}) that for all sufficiently large $k$  and $z, w \in C_1$, the map $\phi_k$  will satisfy the estimate
 $|\phi_k(z)-\phi_k(w)|\geq \frac{1}{2} |z-w|$. In particular, we have

 \begin{lem}\label{phi} For all sufficiently large $k$, the map $\phi_{k}: F^{-1}(\pi)\supset C_1\to F^{-1}(\pi_k)$ defined above is an injective local diffeomorphism
 satisfying $\phi_k(b)=b$. \end{lem}

Fix  global Cartesian coordinates $(x, y)$ in $\pi$,  and let $D\subset \pi$ be the open unit disk centered at $0\in \pi$.  
Without further mention,  from now on it is to be understood that the estimates for elliptic equations from \cite{GT} that we are going to apply refer to this system of coordinates (possibly after composition with some fixed diffeomorphism). \rm

\vskip10pt
\noindent 
Recalling the definition of $U_{\pi}$ and $T_{\pi}$   from (\ref{U and T}), let $\Psi:\overline D \to C_1\setminus {U}_{\pi}$ be a smooth injection such that
\vskip10pt
\noindent (1) $T_{\pi}\cap \Psi(\overline D)=\emptyset$.
\vskip 10pt
\noindent (2) $b\in \Psi(\overline D).$
\vskip10pt
\noindent As pointed out, from (\ref{tending to the  inclusion}) one has $\phi_k(z) \to z$  uniformly for $z\in  C_1$. It now follows from  (1) above, with 
\begin{eqnarray*} \xi_k=\phi_k\circ \Psi:\overline{D} \to F^{-1}(\pi_k), \end{eqnarray*}
that for all sufficiently  large  $k$,

\vskip10pt
\noindent (3) $\xi_k(\overline D) \subset F^{-1}(\pi) \setminus \overline {U}_{\pi_k},$
\vskip10pt
\noindent (4) $ d(T_{\pi_k}, \xi_k(\overline D))\geq c_1>0,$
\vskip10pt
\noindent i.e. the distance between $T_k$ and $\xi_k(\overline D)$ is uniformly bounded below by a positive constant $c_1$ independent of $k$.

Consider now the positive harmonic functions  $u_{\pi_k}|\xi_k(\overline D)$.
In view of (4),   and since  $\xi_k(z) \to \Psi(z)$ uniformly,  the Harnack inequality (\cite {GT}, Corollary 8.21)
applies on $\xi_k(\overline D)$ with a constant $c_2>0$ that is independent  of all sufficiently large $k$:
\begin{eqnarray} \label {Harnack} \max_{\xi_k(\overline D)} u_{\pi_k}\leq  c_2\min_{\xi_k(\overline D)} u_{\pi_k}.  \end{eqnarray}

As $u_{\pi_k}(b)=1$, it follows from (2) above and (\ref{Harnack}) that the restrictions  $u_{\pi_k}|\xi_k(\overline D)$ are uniformly $C^0-$ bounded, independently of $k$, i.e.
 \begin{eqnarray} \label {bounded} \max_{\xi_k(\overline D)} u_{\pi_k}\leq  c_2, \end{eqnarray}
 for an absolute constant $c_2>0$.

Relative to the parametrizations $\xi_k$ of  (portions of) $F^{-1}(\pi_k)$,  the elliptic operators $\Delta_{\pi_k}$ satisfy the estimate
\begin{eqnarray} \label{Schauder} d|Du|_{0:\Omega'}+d^2|D^2 u|_{0:\Omega'}+d^{2+\alpha}||D^2u||_{\alpha:\Omega'}\leq c_3|u|_{0, \Omega}, \end{eqnarray}
where $\Omega'$ is a compact subset of the open set $\Omega$, $u=u_{\pi_k}$ and  $d, c_3$ are independent of all sufficiently large $k$ (see \cite{GT}, Corollary 6.3 with $f=0$,  and the remarks that follow it). 

From (\ref{bounded}) and (\ref{Schauder}),
the first and second derivatives of $u_{\pi_k}$, relative to the coordinates given by $\xi_k$,  are locally uniformly bounded. Hence, these functions $u_{\pi_k}$ and their first derivatives  constitute     equicontinuous families. One can then extract a subsequence of $u_{\pi_k}$ that converges locally uniformly in the $C^1$ norm. 
Since $\xi_k\to \Psi$ uniformly, by elliptic regularity the limit function is $\Delta_{\pi}$-harmonic on $\Psi(\overline D)$.

Next, one covers $C_1 \setminus \overline {U}_{\pi}$ by  countably many sets $\Psi_n(\overline D)$, where the  function
\begin{eqnarray*} \Psi_n:\overline D \to C_1 \setminus \overline U_{\pi} \end{eqnarray*}
plays the role of
$\Psi_1\stackrel {def}=\Psi$ in the preceding discussion; in other words,   the functions $\Psi_n$  satisfy (1) and  (2).

Starting with $\Psi_1$, we  extracted a subsequence of harmonic functions  that converges $C^1$-uniformly over $\Psi_1(\overline D)$.
Taking  the previous  subsequence of positive integers as  a new  starting point, one applies the same argument  in order to obtain a  convergent  subsequence of functions  corresponding to $\Psi=\Psi_2$ that converges $C^1$-uniformly over $\Psi_2(\overline D)$.

Iterating this process for all $\Psi_n$ one obtains a doubly infinite sequence,  and by the familiar argument using the diagonal  there is  a subsequence of functions that converges locally  uniformly to a harmonic function on $\cup_{j=1}^{\infty} \Psi_{n_j}(\overline D)=C_1\setminus \overline {U}_{\pi}$.

We now  consider the sequence obtained above as the starting point, and use the same arguments  replacing $C_1$ by the component $C_2\supset C_1$ of the domain $K_2\setminus F^{-1}(\pi)$ that contains $b$, in order to obtain a subsequence that converges locally uniformly on $C_2\setminus \overline {U}_{\pi}$. 

Keeping in mind that $\{C_m\}_{m=1}^{\infty}$ is an exhaustion of $F^{-1}(\pi)$,  and applying the previous  process   to $C_3, C_4, \dots,$ to obtain subsequences of previous subsequences, and then using the diagonal argument one last time,  one arrives at  a function  $\hat u_{\pi}$ satisfying:
\begin{eqnarray}\label{properties 1} \hat u_{\pi}: F^{-1}(\pi)\setminus \overline {U_{\pi}} \to [0, \infty),\;\;\; \Delta_{\mu} \hat u_{\pi}=0, \hat u_{\pi}(b)=1.\end{eqnarray}

Our next task is to examine the  behavior of   $\hat u_{\pi}$  near the boundary $T_{\pi}$.

\begin{lem}\label {fronteira} $\hat u_{\pi}$  extends continuously to $ F^{-1}(\pi)\setminus {U_{\pi}}$  by setting  $\hat u_{\pi}=0$ on $T_{\pi}$.
\end{lem}
\begin{proof}  We begin by recalling the  estimate for the oscillation  of the solution of an elliptic equation $Lu=0$,  with $u$ continuous up to the boundary of  a domain  
$\Omega \subset \mathbb R^n$ (see \cite {GT}, p. 204, eqn. (8.72) for details):
\begin{eqnarray}\text{osc} \label {oscilacao} (u|_{\Omega \cap B_R}) \leq C\{R^{\alpha}(R_0^{-\alpha} \sup_{\Omega \cap B_0} |u| +k) +\sigma(\sqrt {RR_0})\}.
\end{eqnarray}

Here, $0<R\leq R_0$, $B_0=B_{R_0}(x_0)$,  where $x_0\in \partial \Omega$ is a point where  $\Omega$ satisfies an exterior cone condition, and  $\sigma(R)$ is the oscillation of $u$ on $\partial \Omega \cap B_R(x_0)$:
\begin{eqnarray*} \sigma(R)=\sup_{\partial \Omega\cap B_R(x_0)}u -\inf_{\partial \Omega\cap B_R(x_0)}u.\end{eqnarray*}
 
Also, $C$, $k$  and $\alpha$ are positive constants that, ultimately,  depend only on the dimension,  $R_0$, the coefficients of $L$,  and the exterior cone $V_{x_0}$.

In particular, if $\Omega$ is smooth the boundary of the compact domain $\Omega$ will satisfy a uniform exterior cone condition.  Hence, for a smooth family of operators  and smooth relatively compact domains that are small  perturbations of a fixed operator and a fixed domain,  one can pick     values of $C$ and $\alpha$ in (\ref{oscilacao}) that work for all operators in the family. 

This will be the case in our current application, after using coordinates as before, since the boundary pieces $T_{\pi_k}$ have uniformly bounded geometry,  thus guaranteeing the uniform cone condition.

Again, referring  to section 4 for notation, consider  an Euclidean ball $\widetilde W$ in the codomain of $F$, slightly bigger than $W$,  $\partial W \cap \partial \widetilde W =\emptyset$, with a corresponding neighborhood $\widetilde U$ in the domain of $F$, where $a\in U\subset \overline U \subset \widetilde U$,  so that  $\widetilde{U}$   is mapped diffeomorphically by $F$ onto $\widetilde W$. Set
\begin{eqnarray} \label{ok}\widetilde  T_{\pi_k}= \partial \widetilde U \cap F^{-1}(\pi_k), \;\;\;  \widetilde T_{\pi}=\partial U \cap F^{-1}(\pi).\end{eqnarray}

Hence, for $k$ large but fixed, $T_{\pi_k}$ and $\widetilde T_{\pi_k}$  are the inner and outer boundaries, respectively,  of a closed annulus $A_{\pi_k}  \subset F^{-1}(\pi_k)$. Similar considerations apply to $ T_{\pi}$ and $\widetilde T_{\pi}$, with corresponding annulus $A\subset F^{-1}(\pi)$.

We wish to apply (\ref{oscilacao}) to  $\Delta_{\pi_k}$ on $\Omega=A_{\pi_k}$, considering   the $\Delta_{\pi_k}$-harmonic functions $u_{\pi_k}$ 
near the boundary $T_{\pi_k}$, where the operators have been written in  coordinates given by $\xi_k$  as in the proof of  Lemma \ref{phi}.

Using the previous Harnack inequality argument, keeping in mind that all functions $u_{\pi_k}$ have value 1 at $b$,  and
\begin{eqnarray} \inf_k d(T_{\pi_k}, \widetilde T_{\pi_k})>0 \end{eqnarray}
for all sufficiently large $k$, one sees that   the  values of the functions $u_{\pi_k}$  are bounded on $\widetilde T_{\pi_k}$ (the outer boundary of $A_{\pi_k}$) by a constant independent of $k$. 

Therefore, since $u_{\pi_k}$ vanishes on the inner boundary
$T_{\pi_k}$ of $A_{\pi_k}$, by the maximum principle the function   $u_{\pi_k}$ is  bounded on the entire closed annulus $A_{\pi_k}$ by  a constant independent of all sufficiently large $k$. In particular, the oscillation of the $u_{\pi_k}$ is uniformly bounded on $A_{\pi_k}$.

We use this information in (\ref{oscilacao}) as follows.  Fix a point $x_{0}$ in the inner boundary $ T_{\pi_k}$ of $A_{\pi_k}$.  By the above, we can choose $R_0$ sufficiently small  so that the intersection of
$F^{-1} (\pi_k)\setminus U_{\pi_k}$ and the open ball centered at $x_{0}$ and radius $R_0$ is  disjoint from $\widetilde T_{\pi_k}$ (again, one should be reminded that we are working in coordinates). 

Putting all these observations together, and keeping in mind that $u_{\pi_k}$  vanishes on the inner boundary, and using the definition of $\sigma$, one has  $\sigma\equiv 0$, with   $\Omega=A_{\pi_k}$. Hence  (\ref{oscilacao}) assumes the form
\begin{eqnarray*}
\text{osc} \label {oscilacao2} (u_{\pi_k}|_{S_{\pi_k} \cap B_R}) \leq C'R^{\alpha}, \;\; 0<R \leq R_0,
\end{eqnarray*}
for some absolute constant $C'$ independent of  $k$.

Since $u_{\pi_k}=0$ on $T_{\pi_k}$, this implies that for every $x$ in the \it interior \rm of   $S_{\pi_k} \cap B_R$,  $x$ fixed,
\begin{eqnarray} \label {close to the end} u_{\pi_k}(x)\leq C'R^{\alpha}, \;\; 0<R\leq R_0.\end{eqnarray}

In local coordinates the function $\hat u_{\pi}$, originally defined in the \it open \rm set $F^{-1}(\pi)\setminus \overline {U_{\pi}}$,  was obtained as the local uniform limit of a subsequence of  $(u_{\pi_k})$, and  so
(\ref {close to the end}) implies, for every $x$ in the \it interior \rm of   $S_{\pi} \cap B_R$:   
\begin{eqnarray}  \label {close to the end 2} \hat u_{\pi}(x)\leq C'R^{\alpha}, \;\; 0<R\leq R_0.\end{eqnarray}
It now follows from (\ref{close to the end 2}), by letting $R\to 0$,  that  $\hat u_{\pi}(x)\to 0$ uniformly  as $x\in S_{\pi}$ approaches  $T_{\pi}$. This concludes the proof of Lemma \ref{fronteira}.

\end{proof}

\vskip15pt

We can now finish the proof of  Theorem \ref{parabola}.  By (\ref{properties 1}) and Lemma \ref{fronteira},   we have 

\begin{eqnarray} \label {quasedesapareceu} \hat{u}_{\pi}: F^{-1}(\pi)\setminus \overline{U}_{\pi}\to (0, \infty), \; \Delta_{\pi}\hat{u}_{\pi}=0, \hat {u}_{\pi}(b)=1, \;  
\lim_{p\to T_{\pi}}\hat{u}_{\pi}(p)=0.
\end{eqnarray}

(The function in   (\ref{quasedesapareceu}), although obtained as the limit of positive functions,   is indeed strictly positive in the interior,  by the maximum  principle.)

Observing that $F^{-1}(\pi)\setminus \overline {U_{\pi}}$ is conformal to the punctured open unit disc $B\setminus \{0\}$, after composition with a suitable conformal map that applies $T_{\pi}$ onto the unit circle, we can regard $ \hat u_{\pi}$ as being a positive harmonic function defined on $B\setminus \{0\}$.

By Lemma \ref{B}  (B\^ocher's theorem) this
composed function  is of the form $u(z)=\nu(z)-c\log |z|$, where
$\nu$  is harmonic on  $B$ and $c\geq 0$. We claim that the constant $c$ cannot be zero.

Indeed, if $c=0$, looking back at the surface and then at the Riemann sphere, the function associated to $\hat u_{\pi}$ would solve a Dirichlet problem on the northern hemisphere with  boundary values zero on the equator. By uniqueness one would have  $\hat u_{\pi}\equiv 0$,  contradicting $\hat u_{\pi}(b)=1$.

Hence, the function $u$ above has a logarithmic singularity at $0$ and vanishes at the unit circle. 
Unwinding the above composition, this means that, back in the surface $F^{-1}(\pi)$, 
one has $\lim \hat u_{\pi}(p)=\infty$ uniformly as $|p|\to \infty$ in $\mathbb R^n$,  $p\in F^{-1}(\pi) \setminus \overline {U_{\pi}}.$

This last property and  (\ref{quasedesapareceu}),  show that $\hat u_{\pi}$  satisfies (i)-(iv) in Lemma \ref{terceiro lema}.
By the  uniqueness part of the said lemma, $ \hat u_{\pi}=u_{\pi}$.

Since the convergence in the above arguments is in fact $C^k$ for any $k\geq 1$, we have concluded that given any sequence $\pi_k\to\pi$ there is a subsequence $\pi_{k_l}$ for which
$$\lim_{k_l\to \infty}  \nabla u_{\pi_{k_l}}(b)=\nabla {u}_{\pi}(b),$$
where the gradients are taken in their respective surfaces but are viewed as vectors in $\mathbb R^n$.

It follows that  the full sequence $\big ( \nabla u_{\pi_{k}}(b)\big)$  is bounded and, in fact, has only one limit point, so that 
$$\lim_{k\to \infty}  \nabla u_{\pi_{k}}(b)=\nabla {u}_{\pi}(b).$$

In particular, (\ref{subsequence}) holds. This establishes the continuity of $s$, thus concluding the proof of Theorem \ref{parabola}. 

\qed

\vskip10pt

\begin{rem} \rm The hypothesis in Theorem \ref{parabola} specifies that $F^{-1}(\pi)$  is diffeomorphic  to $\mathbb R^2$ and, as a Riemann surface,  biholomorphic  to $\mathbb C$ rather than to  the open unit disc $D \subset \mathbb C$.     Equivalently, from a Riemannian perspective, $F^{-1}(\pi)$ is diffeomorphic to $\mathbb R^2$ and any bounded function on $F^{-1}(\pi)$ that is harmonic in the Riemannian sense is necessarily constant (a Riemannian manifold with the property that every bounded harmonic function is constant is said to satisfy the Liouville property). 

The  assumption  that  $F^{-1}(\pi)$ is conformal to $\mathbb C$ was used  in a crucial way in section 4.
Indeed,   any doubly-connected Riemann surface  is conformal to some  (possibly degenerate) ring $R_{a}=\{z : a<|z|<1\}$, where $0\leq a <1$.
In the  proof of  Lemma \ref{terceiro lema} one uses the fact that, since $F^{-1}(\pi)$ is conformal to $\mathbb C$,    $F^{-1}(\pi)\setminus  \overline U$ is conformal to a punctured neighborhood of infinity in the Riemann sphere which, in turn, is conformal to the degenerate ring $R_0$.

If,  instead,  $F^{-1}(\pi)$ was conformal to $D$,  then $F^{-1}(\pi)\setminus  \overline U$ would be conformal to $R_a$ for some  $a>0$,   and one would not be dealing with an isolated singularity.
This would preclude us from using  Lemma \ref {segundo lema}.  As a consequence,  it would not be clear even how to  define, in a canonical way, an analogue for   the vector field $s$ in Lemma \ref{quarto lema}. In passing, we remark  that the continuity proof given in this section also required  $F^{-1}(\pi)\setminus  \overline U$ to be conformal to $R_0$,   so that Lemma \ref {B} (B{\^o}cher's theorem) could be applied.

\end{rem}

 \section{Condensers, Euler numbers,  and small fibers}

In this section we use the solutions of certain Dirichlet problems in order to prove Theorem \ref{parabola dois}. A noteworthy  feature of this new construction is that, unlike the one used to prove Theorem \ref{parabola}, it is meaningful on any connected finitely punctured  compact Riemann surface, and not only on surfaces  with genus zero. We will elaborate further on this point in section 8.

We assume, by contradiction,  that $F(p_1)=F(p_2)=F(p_3)=q=0$,  where the points in question are all distinct (composing with a translation, there is no loss of generality in assuming that $q=0$). 

Let $W$ be a small Euclidean ball around $0$ in $\mathbb R^n$ and $U_1$, $U_2$ the corresponding neighborhoods of $p_1$ and $p_2$ that are mapped diffeomorphically onto $W$. Write $U_1^{\pi}$, $U_2^{\pi}$ for the intersections of these neighborhoods with the pre-image $F^{-1}(\pi)$ of a two-dimensional vector subspace $\pi$  that is parallel to some tangent plane of $M^2$, and set  $T_i^{\pi}=\partial U_i^{\pi}$, $i=1,2$.

Using the fact that $F^{-1}(\pi)$ is conformally a punctured sphere,   and  isolated singularities  of bounded harmonic functions are removable (in our case, by  assumption  the topological ends  of $F^{-1}(\pi)$ are conformal to punctured discs), one can argue that   that there exists  a unique $\underline{\rm  bounded}$  harmonic function $u_{\pi}$ on the complement of  the closure of
$U_1^{\pi} \cup U_2^{\pi}$ in $F^{-1}(\pi)$ that extends continuously to $T_1^{\pi}\cup T_2^{\pi}$, and satisfies
$u_{\pi}| T_1^{\pi}=0,  \;\;u_{\pi}| T_2^{\pi}=1$.

In the classical literature, this type of Dirichlet problem is  often referred to as a $\underline{\rm condenser}$. As indicated above, this part of the argument also works when $F^{-1}(\pi)$ is conformal to a finitely punctured connected compact Riemann surface of positive genus.
For an  illustration of the condenser construction in a genus two surface, see Fig.2.

\begin{figure}[!ht]
\vspace{-0.5
cm}
\begin{center}
\includegraphics[scale=0.39]{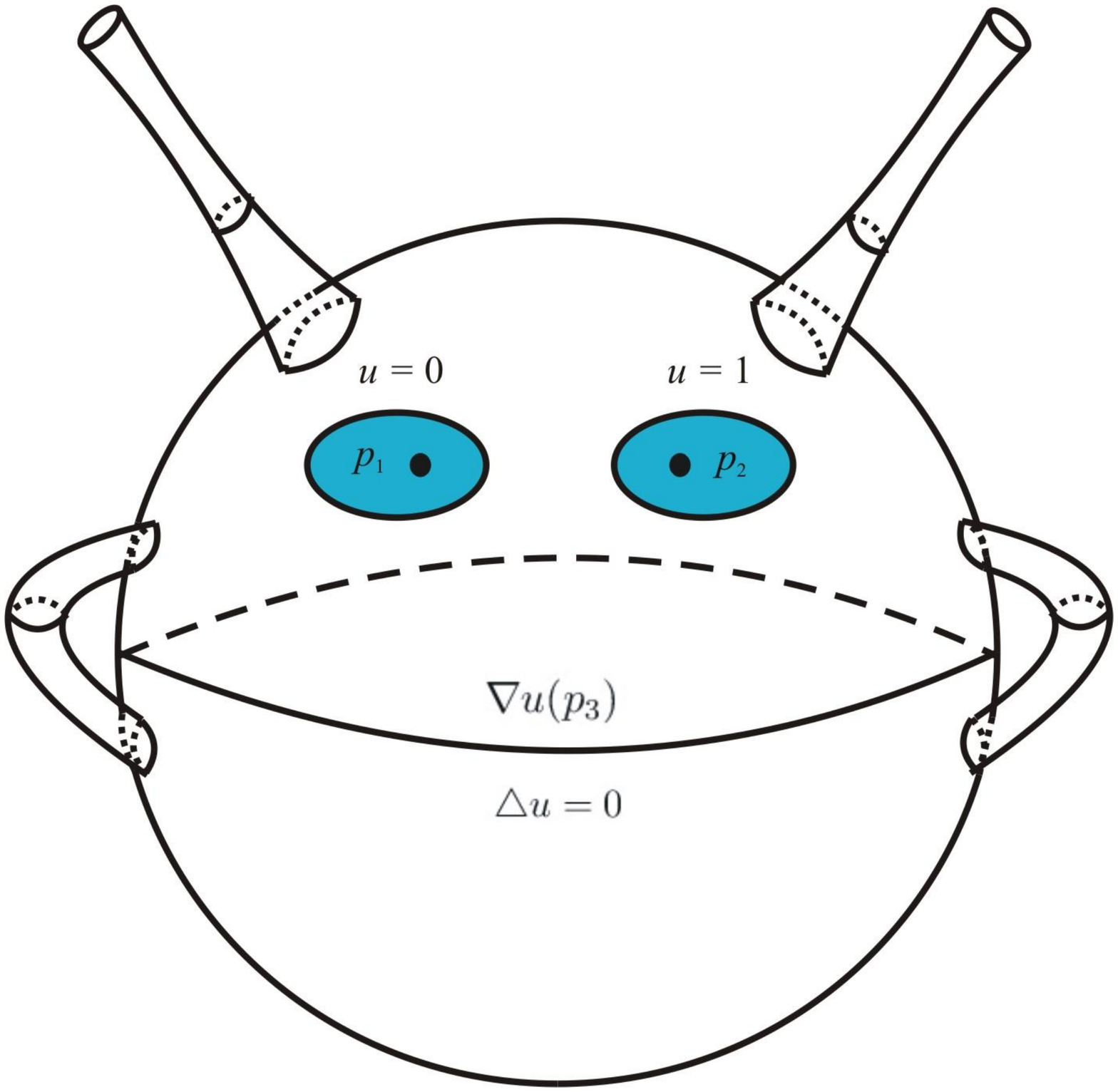}
\vspace{-1.1cm}
\vskip15pt
\caption{\footnotesize {The condenser method when $F^{-1}(\pi)$ is conformal to some twice punctured double torus.}}
\label{Fig:01}
\end{center}
\vspace{-0.5cm}
\end{figure}

Letting $\pi_p=T_pM^2$,    one sees  just as in section 4   that the  assignment  
\begin{eqnarray} \label {secao} M^2\ni p \to  s(p)  = DF(p_3)\nabla u _{\pi_{p}} (p_3)\in T_pM^2 \end{eqnarray}
defines a tangent vector field  on  $M^2$. A contradiction to $\chi(M^2)\neq 0$ will follow as soon as we establish that this vector field is continuous and non-zero.

\vskip10pt
\noindent \bf Proof of Theorem  \ref{parabola dois}.\rm 
\vskip10pt
\noindent The bulk of the proof is quite similar to the technical work done in section 5 to prove Theorem \ref{parabola}.
In the interest of avoiding unnecessary duplication, we keep the same notation wherever possible and only concentrate on  the  steps where the argument needs to be modified. Notice that the definition of $s$ given in (\ref{secao}) is similar 
to (\ref{new def}). 

One needs to establish the analogue of (\ref{subsequence}), where the point $b$ has been replaced by $p_3$. Keeping the same notation as in  
section 5, we let $\{K_n\}$ be an exhaustion of $\mathbb R^n$ by compact sets. Let $C_1$ be the connected component of 
$K_1 \cap [F^{-1}(\pi)\setminus \overline {(U_1^{\pi_z} \cup U_2^{\pi_z})}]$
that contains $p_3=b$. 

One can construct $\phi_k:C_1\to F^{-1}(\pi_k)\subset \mathbb R^n$ as in  (\ref{tilde}).  Lemma \ref{phi} is in force,  and there is  a  smooth injection 
\begin{eqnarray*} \Psi:\overline D \to C_1\setminus {(U_1^{\pi_z}\cup U_2^{\pi_z})} \end{eqnarray*} 
such that  $\Psi(\overline D)$ is disjoint from the boundaries $T_1^{\pi_z}$,  $T_2^{\pi_z}$, and $\xi_k=\phi_k\circ \Psi$
satisfies $\xi_k(\overline D) \subset F^{-1}(\pi)\setminus \overline{(U_1^{\pi_z}\cup U_2^{\pi_z})}$.  Also,  
 $d(T_i^{\pi_k}, \xi_k(\overline D))>c_1>0, \; i=1,2$,  where the constant $c_1$ is independent of all sufficiently large $k$. 

By the maximum principle, $0\leq u_{\pi_k} \leq 1$ everywhere, and so there is no need to argue with the Harnack inequality. It follows directly from  the interior estimate (\ref{Schauder})
that one can extract a subsequence of $(u_{\pi_k})$ so that the induced functions 
 converge locally uniformly in the $C^1$ norm to a $\Delta_{\pi}$-harmonic function on $\Psi(\overline D)$.

Next, one covers $C_1\setminus \overline{(U_1^{\pi_z}\cup U_2^{\pi_z})}$ with  countably many sets $\Psi_n(\overline D)$, where $\Psi_1=\Psi$ and, for $n>1$, $\Psi_n$ has similar properties. Working with convergent subsequences of functions (restricted to a larger set) of subsequences that were previously shown to be convergent  on smaller sets, and then using the diagonal argument, as in the discussion preceding Lemma \ref{fronteira}, one finally obtains a $\Delta_{\pi}$-harmonic function 
$$\hat u_{\pi}:F^{-1}(\pi) \setminus \overline{(U_1^{\pi_z}\cup U_2^{\pi_z})}\to [0,1].$$

Using  the boundary estimates (\ref{oscilacao}), just as before,  one can  argue that $\hat u_{\pi}$ extends continuously to $T_1^{\pi_z}$,  $T_2^{\pi_z}$ with boundary values $0$ and $1$, respectively.  By the uniqueness of the solution of the Dirichlet problem,  $\hat u_{\pi}=u_{\pi}$, and from this  (\ref{subsequence}) follows, where $b$ has been replaced by $p_3$.  Hence  the continuity of $s$ has been established.

If one can argue  that  $\nabla u _{\pi_{p}} (p_3)\neq 0$  for all $\pi_p$, $p\in M^2$, then by ({\ref{secao})  the vector field $s$ would be not only continuous,  but also zero-free, a contradiction to
the Poincar\'e-Hopf theorem since  $\chi(M^2)\neq 0$.

As the ends of $F^{-1}(\pi_p)$ are conformally punctured discs associated  to the removable  singularities of the $\underline{\rm bounded}$ harmonic function $u_{\pi_p}$, and doubly-connected planar regions are conformal to  circular annuli, which are non-degenerate in the present context since they correspond to the complement of 
$U_1^{\pi_p} \cup U_2^{\pi_p}$ in the conformal compactification of $F^{-1}(\pi_p)$, 
we only need to check the non-vanishing of $\nabla u$  when  $u$ is the continuous   function  defined on the closed annulus
$\{z\in \mathbb C: 0<a<|z|\leq 1\}$ that is  harmonic in its interior and satisfies    $u=0$ ($u=1$) in the inner (outer) boundary circle. In this set-up  the boundary circles correspond to $T_1^{\pi_p}$ and $T_2^{\pi_p}$. 
In other words, one only needs to deal with a standard condenser. But  here the solution of the Dirichlet problem is quite explicit, namely:
\begin{eqnarray} \label{condenser} u(z)=\frac{\log a -\log|z|}{\log a}. \end{eqnarray}
Since $\nabla u\neq 0$ everywhere,   $s$   would   be continuous and nowhere zero,  and the proof  of Theorem \ref{parabola dois} is complete.

\qed 

\section{A general theorem  and  the inversion  of  polynomial maps}
 \vskip5pt
\noindent The following points about   Theorems \ref{parabola}  and   \ref {parabola dois}   are worth emphasizing:
\vskip5pt

\noindent a) Instead of endowing the domain $\mathbb R^n$ of $F$ with the standard flat metric, the  proofs  can be easily  modified so as to allow for  $F$ to be  defined on any connected non-compact  manifold $N^n$   carrying  a Riemannian metric that induces on the surfaces $F^{-1}(\pi)$ the  conformal type of  $\mathbb R^2$ in Theorem \ref{parabola}, and that of a   finitely punctured sphere in Theorem \ref{parabola dois}.
\vskip10pt
\noindent b) The  tangent bundle of $M^2$  was used only to the extent that one needed to guarantee that every  continuous section $s$ (i.e. a vector field) must have a zero.

\vskip10pt
\noindent With  remarks a) and b)  in mind, we can state a general  ``abstract"  theorem that  subsumes  both Theorems \ref{parabola} and   \ref {parabola dois},  as well as the algebraic result in \cite{NX1} about the Jacobian conjecture (see Theorem \ref{Jac} below).  The proof of the result below  involves only obvious  adjustments  of the proofs of the above mentioned theorems.
\vskip10pt

\begin{thm} \label{geral}
Let  $N^n$ be a connected  non-compact manifold, $F:N^n \to \mathbb R^n$  a local diffeomorphism,  $n\geq 3$,  $q\in F(N^n)$. Assume the existence of a Riemannian metric $g$ on $N^n$,   and a  smooth  vector bundle $T$  for which the following holds:
\vskip5pt
\noindent i) Every     fiber $\xi$   of $T$  belongs to 
$G_2(\mathbb R^n)$  and  the surface $F^{-1}(q+\xi)$  is  conformally diffeomorphic, with respect  to the metric induced by $g$,  to  a  punctured sphere 
$S^2 \setminus \{p_1, \dots, p_{k(\xi)} \}$.
\vskip5pt
\noindent ii) Every continuous section of  $T$  has a zero.
\vskip5pt
\noindent Then $F^{-1}(q)$  has one or two elements. The first alternative holds if  $k(\xi)=1$ for every $\xi$.
\end{thm}

There is a slight redundancy in the statement of Theorem \ref{geral},  in the sense that the hypothesis $n\geq 3$  follows from  i) and ii). Indeed, if $n=2$  the first half of  i) implies that $T$ is the trivial bundle, contradicting ii).

\vskip5pt

In order to  recover Theorems \ref{parabola}  and   \ref {parabola dois}  from Theorem \ref{geral},    it suffices to take $g$ to be the standard flat metric on $N^n=\mathbb R^n$, and for 
$T$  the tangent bundle of, say,  an  embedded sphere $S^2\subset \mathbb R^n$.  

\vskip10pt

The invertibility criterion in  (JC ) stated  below, and alluded to in Example 2 of section 2,  was  first   proved  in \cite{NX1}:

\begin{thm} \label {Jac}  A  polynomial local biholomorphism $F:\mathbb C^n \to \mathbb C^n$     is invertible if and only if  for  the generic point $q$ in $F(\mathbb C^n)$ the following  holds. For every complex line containing $q$      there is  an open  connected  set $\mathcal O_l\subset \mathbb R^2$ that   is homeomorphic to   $F^{-1}(l) $.
\end{thm}

Due to the polynomial nature of the map, the hypothesis is equivalent to asking that the pre-image of every    complex line through $q$ be a connected rational curve, that is, $F^{-1}(l)$  should be conformal to $\mathbb C\mathbb P^1$ with finitely many points  removed. This was the original formulation in \cite{NX1}.

To see this how Theorem \ref{Jac}   can be recovered from  Theorem \ref{geral}, assume that the pre-image of every complex line  $\pi$ containing $0$,   under a polynomial local biholomorphism $F:\mathbb C^n \to \mathbb C^n$, $n>1$,  is a connected rational curve (there is no loss of generality in assuming that $0$ is a generic point in the image of $F$).

Endowing $\mathbb C^n$ with the flat metric arising from the identification of  $\mathbb C^n$ with the Euclidean space $\mathbb R^{2n}$, any non-singular complex curve
$C \subset \mathbb C^n$ will seemingly have two complex structures, the one inherited from the complex manifold $\mathbb C^n$,   and a second one inherited from  the Riemannian metric induced on $C$ by the standard Euclidean metric on $\mathbb C^n$. Using the Cauchy-Riemann equations and the fact that multiplication by $i$ induces an isometry in $\mathbb C^n$ one can check that these two complex structures coincide.

The complex lines through $q=0 \in F(\mathbb C^n)$ that are contained in a fixed $V\subset \mathbb C^n$, $V\approx \mathbb C^2$,  are  the fibers of the tautological line bundle $T$ over $\mathbb C \mathbb P^1$.  Since $F^{-1}(l)$ is assumed to be connected for every complex line $l\subset V$, and $z\in F^{-1}(l)$  whenever  $F(z)=0$,  the manifold $N=F^{-1}(V)$ is likewise connected. Alternatively, $T$ can be viewed as a real bundle of rank two over $S^2$, so under the hypothesis of Theorem \ref{Jac}, condition i) in Theorem \ref{geral} holds.
Since the Euler number of the tautological bundle  is $-1$, hence non-zero, every continuous section of $T$ must vanish somewhere, so that hypothesis ii)  
also holds.   Therefore, it follows from Theorem \ref{geral} that $\#F^{-1}(0)=1$ or $2$.

Fortunately,  for algebraic reasons  the generic fiber of a polynomial local biholomorphism cannot be $2$ (see Theorem 2.1 of \cite{BCW}, specifically  the equivalence between statements a) and g)).
Hence, $\#F^{-1}(0)=1$, and since $0$  is a generic point $F$  must be injective.  

For  polynomial local biholomorphisms  injectivity implies surjectivity, and so the present  proof of Theorem \ref{Jac} is complete.

In the next section, we will prove a new theorem, similar in spirit to Theorem \ref{Jac}, but involving polynomial  local biholomorphisms with \it real \rm coefficients.

\section {The condenser method in the presence of an involution}

\vskip10pt
In the discussion at the beginning of  section 7, if  we had  assumed that the pre-image of  a plane $\pi$  was conformal to a connected finitely punctured  surface of \it positive \rm  genus, the  Dirichlet problem  would still have a unique bounded solution $u_{\pi}$. 

However, since the genus is positive,  $\nabla u_{\pi}$  would necessarily have zeros. To see this, observe that after filling the punctures conformally the resulting  manifold would be obtained form a compact Riemann surface of positive genus by the removal of  two disjoint small discs. If the gradient of the solution of the Dirichlet problem had no zeros, the  time $t$ map of the flow of the vector  field 
$$X(p)=\frac {\nabla u_{\pi}(p)}{|\nabla u_{\pi}(p)|^2},$$
where the norm comes from any Riemannian metric on the conformal compactification of $F^{-1}(\pi)$,  applies any level set of $u_{\pi}$ onto another level set. Hence,  for all all small $\epsilon>0$ the flow would take  $\{u_{\pi}=\epsilon\}$ onto $\{u_{\pi}=1-\epsilon\}$, and so the compact surface with boundary would have been a topological cylinder $S^1\times [0,1]$, implying that 
the compactification of  $F^{-1}(\pi)$  had genus zero.

Observe that  the zero set of  $\nabla u_{\pi}$ is  discrete. Indeed, locally $u_{\pi}$ is the real part of a holomorphic function $h$ and,  by  the Cauchy-Riemann equations in local coordinates,  $\nabla u_{\pi}$ can be identified with $\overline{h'}$,  with $h$ holomorphic.  Since    $u_{\pi}$ is non-constant, the zeros of $h'$ are discrete, as desired.   

Again referring to section 7, as   the point $p_3$  lies in every $F^{-1}(\pi)$   one cannot  guarantee \it a priori \rm that  $\nabla u_{\pi}(p_3)\neq 0$  as one varies $\pi$ inside of a family of planes. In particular,  if the genus of $F^{-1}(\pi)$ is positive it  cannot be asserted  that   the (continuous) vector field defined in (\ref{secao}) is nowhere zero.

Nevertheless,  as we shall see below, one can push  further  the conceptual connection between the Jacobian conjecture  and the  condenser method. In order to explain this,
let us  take $n=2$ and consider   a  polynomial local biholomorphism $F:\mathbb C^2 \to \mathbb C^2$. If $F$ were to be a counterexample to (JC), the generic point would be covered at least three times 
\cite[Theorem 2.1, equivalence a) $\Longleftrightarrow$  g)]{BCW}.

For notational simplicity, assume that $0$ is a generic point in the image of  $F$.
It is known that there are finitely many complex lines through $0$, say $l_1, \dots, l_k$, with the property  that the pre-image of every complex  line $l\neq l_j$    is a \it connected \rm  finitely punctured  compact Riemann surface of some fixed genus $g\geq 0$.

As remarked above,  in the absence of injectivity the value $0$ would be  covered three or more times.  One can then use the condenser construction of section 7  to create  a  section
\begin{eqnarray*} l \to s(l)=DF(p_3)\nabla u_{l}(p_3) \in l, \end{eqnarray*}
of the restriction of the tautological line bundle to   the complement in $\mathbb C \mathbb P^1$ of the set $\Sigma=\{l_1, \dots, l_k\}$ of exceptional lines. The analytic arguments from section 6 show that $s$ is continuous.

Multiplying $s$ by a smooth  function  on $\mathbb C\mathbb P^1$ that vanishes of a sufficiently high order at $\Sigma$, but is otherwise positive, one obtains  a continuous section $\hat s$  of  the actual tautological line  bundle.

The zeros  of $\hat s$ would occur at $\Sigma$, whenever  this set is non-empty, as well as at the possible zeros of the original section $s$ constructed via  the  condenser method   that might arise if the genus of the generic pre-image is positive, as explained before.

Assuming that  $s$ has  only finitely many zeros,  the sum of the local indices of the singularities (zeros)  of $\hat s$ would be $-1=$ the Euler number of the tautological bundle  
(see \cite[p.133]{H} for a  general version of the Poincar\'e-Hopf theorem). 

Progress in (JC) could conceivably be achieved  if one had enough  information about  these local indices,   so as to conclude that their sum could not be $-1$. 

This contradiction would ensure that the pre-image of the generic point $0$ has at most two elements (recall that  the condenser construction requires at least three points in the pre-image).  
But, as observed before \cite[Thm.2.1]{BCW},  the cardinality of the generic fiber  cannot be two, for algebraic reasons, and  so the polynomial local biholomorphism would be injective.

For $n=2$ the proof of Theorem \ref{Jac}  follows   precisely the  outline given above,  when the scenario  is  the simplest one possible.
Indeed, under the hypotheses of Theorem \ref{Jac}
the pre-images of {\it all}  complex lines are  connected, so the exceptional lines $l_j$  would be absent from the  above discussion.  
Also, since the pre-image of each  complex  line is  supposed to be conformally a finitely punctured $\mathbb C \mathbb P^1$, the gradient of the harmonic function in the condenser construction
has no zeros (see the discussion preceding (\ref{condenser})).
Hence,  $\hat s$ would  be zero-free, an outright contradiction to the fact that the Euler number of the tautological bundle is non-zero.

In the absence of a technique for computing  these local indices, we explore  below a different line of argument, familiar in  topological problems that can be approached  using degree theory,    where ``symmetry" and ``parity" replace the need to actually compute or estimate the local   indices.

Specifically,  since the Euler number of the tautological line bundle is $-1$, a contradiction also  ensues in the preceding discussion if  $\hat s$  is known to have  an even number of  singularities, grouped in pairs,  whose  indices are  either pairwise equal or differ by a sign. Indeed, the sum of the indices would be zero mod (2). The preceding   discussion  provides   the  heuristics for  the   main result of this section.

It will be observed in Lemma \ref{outside Sigma} that if $F:\mathbb C^n \to \mathbb C^n$, $n\geq 2$, is a polynomial local biholomorphism with $\underline{\rm real}$  coefficients, and $\Sigma \subset \mathbb C^n$ is any complex algebraic hypersurface, then $F(\mathbb R^n)$  is not contained in $\Sigma$. This allows us to make sense of the expression ``a generic point in $F(\mathbb R^n)$", which is part of the hypotheses of the theorem below.

\begin{thm} \label {newJac}  A  polynomial local biholomorphism $F:\mathbb C^n \to \mathbb C^n$  that commutes with conjugation    is invertible if and only if  for  the generic point $q\in F(\mathbb R^n)\subset \mathbb R^n$ the following  holds: For every complex line $l$ containing $q$    that is invariant  under   conjugation  there is  an open  connected  set $\mathcal O_l\subset \mathbb R^2$ such that   $F^{-1}(l)$ is homeomorphic to   $\mathcal O_l$.
\end{thm}

 Here,   it is meant that the complex line $l\subset \mathbb C^n$  is invariant  under   conjugation as a set,  and not in the pointwise sense.  In $\mathbb C^2$ such a line is   the zero set of a linear equation with real coefficients. Commutation of $F$ with conjugation follows if the components of $F$ are polynomials with real coefficients coefficients.

 Since $F$ is a polynomial map, the topological assumption in the statement of the theorem implies that $\mathcal O_l$ is conformal to $\mathbb C \mathbb P^1$ punctured a finite number $n_l$ of times, i.e. $F^{-1}(l)$ is a connected rational curve.

\begin{rem} \rm Theorem \ref{newJac} invites comparison with  Theorem \ref{Jac}. In the latter,  injectivity is achieved if the pre-images of  all    complex lines $l$  through a generic point of $F(\mathbb C^n)$ are homeomorphic to   planar domains. 
In Theorem \ref{newJac}, where the coefficients of the map are real, accordingly,  in order to obtain injectivity one only needs  that the pre-images of  all  complex lines given by linear equations with real coefficients, and passing through a generic point in $F(\mathbb R^n)$,  be homeomorphic to planar domains. 

One should stress  that the hypotheses of Theorem \ref{newJac} impose no restrictions  on the topology of the pre-images of those complex lines that are not  invariant  under complex conjugation.

\end{rem}

\begin{rem} \rm  As observed by several authors, the Jacobian conjecture holds  if and only if injectivity can be established  in all dimensions for those   polynomial local biholomorphisms of 
 $\mathbb C^n$ into itself that have $\underline {\rm real} $ coefficients.

Indeed, 
if $F:\mathbb C^n \to \mathbb C^n$ is polynomial and $\text{det} DF=1$,  its  realification   $\widehat F:\mathbb R^{2n} \to \mathbb R^{2n}$   satisfies
$$ \text{det} D\widehat F(p)=|\text{det} DF(p)|^2=1.$$
(More generally, this relation holds for all holomorphic maps $F:\mathbb C^n \to \mathbb C^n$.) The  complexification $\widetilde F:\mathbb C^{2n}\to \mathbb C^{2n}$ of $\widehat F$ has real coefficients and
 satisfies $\text{det}D\widetilde F=1$. Clearly, if $\widetilde F$ is injective so  are, in succession,  $\widehat F$ and  $F$. 
 
Note, however, that  Theorems \ref{Jac} and \ref{newJac} are logically independent invertibility criteria.

\end{rem}

Before commencing  the proof of Theorem \ref{newJac},  we make some elementary observations that will also help us  fix the  notation. If $\mathbb C^2$  is given coordinates $(z,w)$, the complex lines through $0$ are of the form $z=0$ and $w=\alpha z$, with $\alpha \in \mathbb C$, and so they are  identified with the points in $\mathbb C \mathbb P^1$ given in homogeneous coordinates by $(0:1)$ and $(1:\alpha)$, $\alpha \in \mathbb C$. 

The tautological line bundle $T$, when restricted to the affine part of $\mathbb C \mathbb P^1$, associates to $\alpha \in \mathbb C$ the complex line given by $w=\alpha z$. 
Hence, $\xi(\alpha)=(1, \alpha)$, $\alpha \in \mathbb C$,  defines a section of the restriction $T|\mathbb C$.  The Euler number $e(T)$ of the tautological line bundle  is  the local index of $\xi$ at its singularity at infinity. 

Introducing a new coordinate $\zeta= 1/\alpha$, $|\zeta|\leq 1$, the index at infinity can be computed as the local index at $0$ of the section 
$\zeta \to (1, 1/\zeta)=(1, \overline {\zeta}/|\zeta|^2)$,  and so $e(T)=-1$.

Complex conjugation  acts not only on $\mathbb C^2$,  but  on $\mathbb C \mathbb P^1$ as well, via 
$\overline {(z:w)}=(\overline {z} : \overline {w})$. 
The real projective line $\mathbb R \mathbb P^1$ can be identified with  the maximal subset  of $\mathbb C \mathbb P^1$ that is left pointwise invariant under conjugation, namely
$$ \mathbb R \mathbb P^1\cong \{(1:x) \;|\; x \in \mathbb R\} \cup \{(0:1)\}.$$
Indeed, if $\overline{(z:w)}=(z:w)$, with $z\neq 0$, one has
$ \frac{w}{z}=\frac{\overline w}{\overline z}=\overline{(\frac{w}{z})}=x\in \mathbb R.$
In particular,  $(z:w)=(z:x z)=(1: x).$

Identifying $\mathbb C \mathbb P^1$ with the unit  sphere $S^2$ given by $x^2+y^2+t^2=1$  in $\mathbb R^3$ endowed with coordinates $(x, y, t)$,  $z=x+iy$, and then using stereographic projection from the north pole $(0,0, 1)$,  one sees that $\mathbb R \mathbb P^1$ can be identified with the intersection of $S^2$  and  the plane  $y=0$. 
Hence, 
$\mathbb C \mathbb P^1\setminus \mathbb R \mathbb P^1= D^{+} \cup D^{-},$
where  $D^{+}$ and $D^{-}$ are disjoint topological  discs, intertwined  under conjugation,  corresponding in the affine part of $\mathbb C\mathbb P^1$ to  the sets $\{z\in \mathbb C \;| \text{Im} z >0\}$ and  $\{z \in \mathbb C \; | \text{Im} z <0\}$.

Besides the ordinary \it complex  \rm tautological complex  line bundle $T$ over $\mathbb C \mathbb P^1$, one has in a  natural way a tautological \it  real \rm  line bundle over $\mathbb R \mathbb P^1$. 
Namely, one associates to $(0:1)$ and $(1:\beta)$, where $\beta\in \mathbb R$, the lines in $\mathbb R^2$ given by $x=0$ and $y=\beta x$.  
We denote  this real line bundle   by $T_{\mathbb R}$. 
Abusing the notation somewhat,  we bypass the  canonical projection and,  for a given complex line $l$ in $\mathbb C^2$ that passes  through $0$, we write both 
$l\subset \mathbb C^2$ and  $l\in \mathbb C \mathbb P^1$. 
Let $T_0$ be the restriction of $T$ to $\mathbb R \mathbb P^1$.  

It is clear that a section $\xi_0$ of $T_0$ that satisfies $\xi_0(l)=\overline{\xi_0(l)}$ for every $l\in \mathbb R \mathbb P^1$ induces a section of $T_{\mathbb R}$,  and vice versa. Since $T_{\mathbb R}$ is a real line bundle over $\mathbb R \mathbb P^1\approx S^1$, it is isomorphic to either the trivial bundle or the   M\"obius bundle. It follows from the lemma below that the second alternative holds.

\begin{lem}  \label{conjugation twice} Every continuous section of $T_{\mathbb  R}$ vanishes somewhere in $\mathbb R \mathbb P^1$.  Otherwise said, if   $\xi_0$  is    a continuous section of $T_0$  that satisfies 
$\xi_0(l)=\overline {\xi_0(l)}$ for all  $l\in \mathbb R \mathbb P^1$, then 
 $\xi_0$ must have a zero on $\mathbb R \mathbb P^1$.
 \end{lem}

To prove this lemma, suppose  by contradiction that $\xi_0$ vanishes nowhere on $\mathbb R \mathbb P^1$. 
One can  extend  $\xi_0$ continuously to a section $\xi_0^{+}$ of the restriction  $T|D^{+}$,  smooth outside a half-collar of 
$\mathbb R \mathbb P^1$ in $D^{+}$, and which,  by transversality, has  only finitely many zeros on $D^{+}$. 
Having obtained $\xi_0^{+}$, we now define a section $\xi_0^{-}$  on $D^{-}$ by ``reflection", setting
\begin{eqnarray}  \label{reflection} \xi_0^{-}(l)=\overline {\xi_0^{+}(\overline l)}, \;\; l\in D^{-}.\end{eqnarray}
Consider the  section $\xi$ of $T$ defined to be  $\xi_0^{+}$  on $D^{+}$, $\xi_0^{-}$  on $D^{-}$, and $\xi_0$ on $\mathbb R\mathbb P^1$.
Since $\mathbb R \mathbb P^1$ is kept pointwise  fixed under conjugation,  the hypothesis of the lemma and (\ref{reflection}) imply that $\xi$ is continuous. Since we are assuming   that $\xi_0$   has no zeros on $\mathbb R\mathbb P^1$, the zeros of 
$\xi$, if any, must  lie in $D^{+}\cup D^{-}$  and  occur in pairs, one in $D^{+}$ and the other in $D^{-}$.  
One can check using  (\ref{reflection}) that     the   local indices  are pairwise the same.  

To see this, denote by $\Lambda$ the action of complex conjugation on $\mathbb C \mathbb P^1$, so that $\Lambda (l)=\overline l$.
One can rewrite (\ref{reflection}) as 
\begin{eqnarray} \label {rewriting} \xi_0^{-}=\Lambda \circ \xi_0^{+}\circ \Lambda. \end{eqnarray}
If $l_0$ is a zero of $\xi_0^{-}$ then $\Lambda (l_0)$ is a zero of $\xi_0^{+}$, and vice versa. 
Keeping in  mind that in the natural local coordinates $\Lambda$ is a real linear involution, so that $\Lambda^{-1}=\Lambda$ and $d\Lambda(l_o)=\Lambda$, 
one can rewrite  (\ref{rewriting}) as
\begin{eqnarray} \xi_0^{-}(l)=\Lambda(\xi_0^{+}(\Lambda (l)))=\big[d\Lambda(\l)\big ]^{-1}(\xi_0^{+}(\Lambda (l)).
\end{eqnarray}
The last relation displays  the section $\xi_0^{-}$  as  the pull-back  of $\xi_0^{+}$  under $\Lambda$. In particular,
\begin{eqnarray} \text{Ind} (\xi_0^{-}, l_0)= \text {Ind}(\xi_0^{+}, \Lambda (l_0)).\end{eqnarray}

Hence, as claimed, the zeros of $\xi$ occur in pairs,  and the zeros  in a given pair have   the same local index. 
It follows that the sum of all local indices is $0$ mod (2),  contradicting the fact that  the Euler  number of $T$ is $-1$. This concludes the proof of the lemma.

\vskip10pt

To prove Theorem \ref{newJac} we will argue by contradiction, using  the solutions of the Dirichlet problems  from section 6 in order to create a  section $\xi_0$ of $T_{\mathbb R}$ that violates Lemma \ref{conjugation twice}. But before we can achieve this, we need to prove one more lemma. For the sake of completeness, we include its proof.

\vskip10pt
\begin{lem} \label {outside Sigma} Let $F:\mathbb C^n \to \mathbb C^n$ be a polynomial local biholomorphism with real coefficients, $n\geq2$,  and $\Sigma$ a complex hypersurface. 
Then $F(\mathbb R^n) \cap [\mathbb R^n \setminus \Sigma]\neq \emptyset$. \end{lem}

\begin {proof} Assume by contradiction that  $F(\mathbb R^n) \subset \Sigma$, so that
\begin{eqnarray} \label {inclusion} \mathbb R^n \subset \Sigma_1\stackrel{def}=F^{-1}(\Sigma). \end{eqnarray}

Let $\Sigma_2$ be the singular set of the complex algebraic variety  $\Sigma_1$, $\Sigma_3$ the singular set of $\Sigma_2$, and so on. One then has a descending chain of inclusions
\begin{eqnarray} \label {chain} \Sigma_1\supset \Sigma_2\supset \Sigma_3 \supset \cdots \supset \Sigma_{r-1}\supset \Sigma_r, \end{eqnarray}
where $r\geq 1$ and $\Sigma_r$ is  smooth and closed  in $\mathbb C^n$. 
From $(\ref{inclusion})$ and $(\ref{chain})$ we can write
\begin{eqnarray} \mathbb R^n=\mathbb R^n \cap \Sigma_1=\big[ \mathbb R^n\cap (\Sigma_1\setminus \Sigma_2)\big]  
\cup \cdots \cup \big[ \mathbb R^n\cap (\Sigma_{r-1}\setminus \Sigma_r)\big] 
\cup\big[ \mathbb R^n\cap \Sigma_r\big ].
\end{eqnarray}
By the Baire category theorem, the closure in $\mathbb R^n$ of at least one of the sets in the brackets on the right hand side  of the above decomposition has non-empty interior in $\mathbb R^n$. 

We distinguish between two cases (notation: the symbol $\text{int}$ refers to the operation of taking the interior of a set in $\mathbb R^n$, whereas the bar stands for the operation of taking closure in $\mathbb R^n$):
\vskip10pt
\noindent a) $\text{int}\overline { (\mathbb R^n\cap \Sigma_r])} \neq \emptyset$.
\vskip10pt
\noindent b) $\text{int} \overline {\big[ \mathbb R^n\cap (\Sigma_{i}\setminus \Sigma_{i+1})\big]} \neq \emptyset$, for some $i=1, \dots, r-1$.

\vskip10pt
\noindent We treat  b) first. Choose $j$ to be the largest index $\leq r-1$ for which

$$\text{int} \overline {\big[ \mathbb R^n\cap (\Sigma_{j}\setminus \Sigma_{j+1})\big]} \neq \emptyset.$$
Since $\mathbb R^n$, $\Sigma_j$ are both closed in $\mathbb C^n$, and $\Sigma_{j+1}$ has empty interior in $\Sigma_j$, the last condition and the maximality of $j$ imply
\begin{eqnarray} \label {minimum j} \text{int} [\mathbb R^n\cap\Sigma_{j}] \neq \emptyset, \;\;\;\; \text{int} [\mathbb R^n\cap\Sigma_{j+1}] = \emptyset. \end{eqnarray}
By the first half of (\ref{minimum j}) there exist $z\in \mathbb R^n\cap\Sigma_{j}$ and $\epsilon>0$  such that the open ball $B(z, \epsilon)$ in $\mathbb R^n$ satisfies   
$B(z,\epsilon)\subset \Sigma_{j}.$

If  $z\in\Sigma_{j}\setminus \Sigma_{j+1}$, set $z'=z$ and choose $\epsilon' $ such that $0<\epsilon'<\min\{ \epsilon, d(z', \Sigma_{j+1})\}$,  where $d$ 
stands for the Euclidean distance, so that $B(z', \epsilon')\subset \Sigma_{j}\setminus \Sigma_{j+1}$. 

If  $z\in \Sigma_{j+1}$,  then by  the second  half of (\ref{minimum j}) one has 
$\text{int} [B(z, \epsilon) \cap \Sigma_{j+1}]= \emptyset,$
in which case there exists 
$z'\in B(z, \epsilon) \cap (\Sigma_{j}\setminus \Sigma_{j+1})$ and some $\epsilon'>0$ such that $B(z', \epsilon')\subset \Sigma_{j}\setminus \Sigma_{j+1}$.

Thus, in both cases above the point $z'$ is the center of a ball in $\mathbb R^n$ that is fully contained in the regular set $\Sigma_{j}\setminus \Sigma_{j+1}$. 

Looking at tangent spaces, one has $\mathbb R^n \subset T_{\mathbb C, z'} (\Sigma_{j}).$  But $T_{\mathbb C, z'} (\Sigma_{j})$ is invariant under multiplication by $i$, so that 
$\mathbb C^n =\mathbb R^n +i\mathbb R^n \subset T_{\mathbb C, z'} (\Sigma_{j})$, implying that  
\begin {eqnarray*} n=  \text {dim} \Sigma_j\leq \text {dim} \Sigma_1=n-1, \end{eqnarray*} a contradiction. This completes the proof of Lemma \ref{outside Sigma} under alternative b). 

If alternative a) is in force one has $\text{int} (\Sigma_r)\neq \emptyset$,  and since $\Sigma_r$ is smooth  we also obtain a ball $B(z', \epsilon')$ in $\mathbb R^n$ that is fully contained in the non-singular variety $\Sigma_r$.  The  argument in the previous paragraph, involving tangent spaces,  leads to  $n=  \text {dim} \Sigma_r\leq \text {dim} \Sigma_1=n-1$, a contradiction. This concludes the proof of Lemma \ref{outside Sigma}.

\end{proof}

It is well known that under the hypotheses of the Jacobian Conjecture  there exists a complex algebraic hypersurface $\Sigma \subset \mathbb C^n$, perhaps with singularities, for which  the restriction 
$\mathbb C^n\setminus F^{-1}(\Sigma)\stackrel{F}{\to} \mathbb C^n\setminus \Sigma$
is a cover map.  

Continuing with the proof of Theorem \ref{newJac}, by Lemma \ref{outside Sigma} we can choose a  point $q$  in   $\mathbb R^n\setminus \Sigma$ of the form $q=F(p_3)$, where 
$\overline p_3=p_3$, i.e. $p_3\in \mathbb R^3$. 

Let $V=\mathbb C^2\subset \mathbb C^n$ (natural inclusion), and $V'=q+V$. If $l$ is a complex line through $q\in \mathbb R^n$ that is contained in $V'$ and is invariant under conjugation then, by hypothesis, $F^{-1}(l)$ is connected and contains $p_3$.  
In particular, $N=F^{-1}(V')$ is connected as well.  

Suppose, by way of contradiction, that $F$ is not injective. Since, as observed in section 7,  the generic fiber of $F$ must have  cardinality at least three,  and $q\notin \Sigma$, there are points $p_1, p_2  \in N$ such that $p_1$, $p_2$, $p_3$ are all distinct and 
$F(p_j)=q$, $j=1,2,3$. If $F(z)=q$, then $F(\overline z)=\overline{F(z)}=\overline{q}=q$.  Thus, one can assume that:

\vskip10pt
\noindent   ($\bullet$) \it Either both  $p_1$ and $p_2$  belong to $ \mathbb R^n$, or both $p_1$ and $p_2$  belong to $ \mathbb C^n \setminus \mathbb R^n$  and $\overline p_1=p_2$.\rm
\vskip10pt

Let $W$ be an  Euclidean  ball in $\mathbb C^n$, centered at  $q$,  and $U_1, U_2$ the corresponding disjoint neighborhoods of $p_1$ and $p_2$ that are mapped biholomorphically onto $W$.

The complex lines $l$ in $V'$ that contain $q$ can be naturally identified with $\mathbb C\mathbb P^1$.  Likewise, recalling that $q\in \mathbb R^n$, those complex lines in $V'$ that contain $q$ and  are invariant under conjugation are naturally identified with $\mathbb R\mathbb P^1\subset \mathbb C \mathbb P^1$. 

Given  $l\in \mathbb R \mathbb P^1$, write $U_1^l, U_2^l$ for the intersections of the neighborhoods  $U_1, U_2$ with $F^{-1}(l)$, 
and set $T_i^l=\partial U_i^l$ (the boundary is taken in $F^{-1}(l)$), $i=1,2$. 

Arguing as in section 6, there is a unique bounded harmonic function $u_l$ defined in the complement of the closure of $U_1^l \cup  U_2^l$ in the finitely punctured compact Riemann surface $F^{-1}(l)$ that extends continuously to $T_1^l\cup T_2^l$ and satisfies 
\begin{eqnarray*}u_l|T_1^l=0,\;\;\;u_l|T_2^l=1.\end{eqnarray*}

\noindent As observed in ($\bullet$), either $p_1,\;p_2\in \mathbb R^n$, or $p_1,\;p_2\in \mathbb C^n \setminus \mathbb R^n, \overline p_1=p_2$. We will show that

\begin{lem} \label {primeiro} If the radius of the Euclidean ball $W$ is sufficiently small and $l\in \mathbb R \mathbb P^1$,  then $F^{-1}(l)=\overline{ F^{-1}(l)}$. Furthermore,
\vskip5pt 
\noindent a) If $p_1, p_2 \in \mathbb R^n$, then \begin{eqnarray}  \label{conjugation} U_j^l=\overline {(U_j^l)}, \;\;\; T_j^l=\overline{(T_j^l)}, \;\; \; j=1,2.\end{eqnarray}
\vskip5pt
\noindent b) If  $p_1, p_2 \in \mathbb C^n \setminus \mathbb R^n$, $\overline p_1=p_2$, then  \begin{eqnarray} \label{another conjugation} U_1^l=\overline {(U_2^l)}, \;\;\; T_1^l=\overline{(T_2^l)}.\end{eqnarray}
\end{lem}

We proceed with the proof of  Lemma \ref{primeiro}.  
By hypothesis, $F(\overline z)=\overline{F(z)}$ ($F$ has real coefficients) and $\overline l=l$. Also, we chose $W$ to be an Euclidean ball centered at $q\in \mathbb R^n\setminus \Sigma$, $W\cap \Sigma=\emptyset$, so that $W=q+W_0$ where  $W_0$ is an Euclidean ball centered at  $0$. In particular, since conjugation is an Euclidean isometry,   $\overline {W}_0=W_0$.  Hence,  since $q\in \mathbb R^n$, one has
\begin{eqnarray*} \overline W= \overline{(q+W_0)}=\overline q + \overline {W_0}=q+W_0=W. \end{eqnarray*}
In order to establish  the relation $F^{-1}(l)=\overline{ F^{-1}(l)}$, observe
\begin{eqnarray*} z\in F^{-1}(l) \Leftrightarrow \overline {F(z)}\in \overline l\Leftrightarrow  F(\overline z)\in l \Leftrightarrow  \overline z \in F^{-1}(l)\Leftrightarrow  z\in \overline{F^{-1}(l)}. \end{eqnarray*}
Since $\mathbb C^n\setminus F^{-1}(\Sigma)\stackrel{F}{\to} \mathbb C^n\setminus \Sigma$ is a finite covering, say of degree $N$, if the radius of $W$ is small enough then 
\begin{eqnarray} \label {covering}F^{-1}(W\cap l)=\bigcup_{n=1}^{N} U_n^l,\end{eqnarray}
where the $U_n^l$ are pairwise disjoint and are applied biholomorphically by $F$ onto $W\cap l$.
To establish  $U_1^l = \overline {U_1^l}$   in (\ref{another conjugation}), and similarly for $U_2^l$,   observe  that 
\begin{eqnarray} \label{infinite union} z\in U_1^l \Rightarrow F(\overline z) =\overline{F(z)}\in \overline {F(U_1^l)}=\overline {W\cap l}=W\cap l 
\Rightarrow \overline z \in F^{-1}(W\cap l)= \bigcup_{n=1}^{N} U_n^l.
\end{eqnarray}
For each $n$, the set $ A_n
=\{z\in U_1^l :\overline z\in U_n^l\}$ is open in $U_1^l$, $A_n\cap A_m=\emptyset$ if $n\neq m$ and, by (\ref{infinite union}), $U_1^l=\cup_{n=1}^{N}A_n$. 
By connectedness, there exists a unique $n\geq 1$ such that $U_1^l=A_n$, so that  $U_1^l=\overline{U_n^l}$. Similarly, there is a unique  integer $m$ such that 
$U_2^l=\overline{U_m^l}$.
Since  $p_1 $, $p_2$  belong to $\mathbb R^n$, $U_1^l\cap \overline{ U_1^l}\neq \emptyset$ and  $U_2^l\cap \overline{ U_2^l}\neq \emptyset$. Hence, $n=m=1$ and so
$U_1^l=\overline{U_1^l}$,   $U_2^l=\overline{U_2^l}$. The proof of $T_j^l=\overline{(T_j^l)}, \;\; \; j=1,2$ runs along similar lines. This concludes the proof of  a). 

For the proof of b)  one follows the above arguments  up to the existence of a unique $n$ such that $U_1^l=\overline{U_n^l}$. As we are assuming $p_1, p_2 \in \mathbb C^n \setminus \mathbb R^n$, $\overline p_2=p_1$, the intersection $U_1^l \cap \overline{U_2^l} $ is non-empty, and therefore one has $U_1^l=\overline {(U_2^l)}$. In a similar way one obtains $T_1^l=\overline{(T_2^l)}.$  This concludes the proof of Lemma \ref{primeiro}.

\vskip10pt

To continue with the proof of Theorem \ref{newJac}, assume first  that $p_1, p_2\in \mathbb R^n$ and for $l\in \mathbb R \mathbb P^1$ consider  the function
\begin{eqnarray*} v_l: F^{-1}(l)\setminus (U_1^l\cup U_2^l), \;\;  v_l(z)=u_l(\overline z).\end{eqnarray*}

That  $v_l$ is well-defined follows from (\ref{conjugation}). Writing  $z=\phi(\zeta)\in  F^{-1}(l)\setminus \overline{(U_1^l\cup U_2^l)}$, where  $\phi$ is a  local holomorphic coordinate, and observing that the composition of a harmonic function with an anti-holomorphic function is still harmonic, one sees that $v_l$ is harmonic in the interior of its domain. 

 Using (\ref{conjugation}) again, one sees that $v_l$  extends continuously to the boundary  with  values  $0$ and $1$  on $T_1^l$ and $T_2^l$, respectively. 
 By the uniqueness of the solution of the Dirichlet problem, it follows  that $v_l=u_l$. Hence, for all $z$ in the domain of $u_{l}$, i.e. $F^{-1}(l)\setminus  {(U_1^l\cup U_2^l)}$, one has 
\begin{eqnarray} \label{first u} u_l(z)=u_l(\overline z).\end{eqnarray} 

If $\overline l=l$ and $\beta:(-\epsilon , \epsilon) \to F^{-1}(l)\setminus \overline{(U_1^l\cup U_2^l)}$ is a smooth curve with $\beta(0)=z$ and $\beta'(0)=w$, differentiation of 
$u_l(\beta(t))=u_l(\overline {\beta (t)})$ yields $du_l(z)w=du_l(\overline z) \overline w.$ 
Equivalently, since conjugation is an Euclidean isometry, 
\begin{eqnarray*} \langle \nabla u_l(z), w\rangle= \langle \nabla u_l(\overline z), \overline w \rangle =\langle \overline{\nabla u_l(\overline z)}, w\rangle. \end{eqnarray*}

As  $w$ is arbitrary, this implies $\nabla u_{l}(z)=\overline{\nabla u_l(\overline z)}$, and so 
\begin{eqnarray}\label{nabla} \overline {\nabla u_{l}(z)}={\nabla u_l(\overline z)}, \;\;\; \overline l=l. \end{eqnarray}

Set  $\xi_0(l)=DF(p_3)\nabla u_l(p_3)$, $l\in \mathbb R \mathbb P^1$. It follows from the proof of Theorem \ref{parabola dois} in section 6 that $\xi_0$ is continuous. Also, since  $F^{-1}(l)$ is conformal to a finitely punctured $\mathbb C \mathbb P^1$,  
 $\xi_0$ is nowhere zero.

Recalling   that  $\overline p_3=p_3$ and  $\overline l=l$, we have, after   using (\ref{nabla})  and the fact that the matrix $DF(p_3)$ has real entries because the components of $F$ have real  coefficients:
\begin{eqnarray*}
\overline {\xi_0(l)}=\overline{DF(p_3)\nabla u_{ l}(p_3)}=\overline{DF(p_3)} \;\; \overline {\nabla u_{l}(p_3)}=DF(p_3)\nabla u_l(\overline{p_3})=DF(p_3)\nabla u_l(p_3)=\xi_0(l). 
\end{eqnarray*}
As $\xi_0$ is continuous and nowhere zero,  this contradicts   Lemma \ref{conjugation twice}. This concludes the proof of  Theorem \ref{newJac} under the assumption that $p_1, p_2\in 
\mathbb R^n$. 

Let us now deal with the remaining alternative, namely $p_1=\overline {p_2}$. In this case, conjugation switches  the $U's$ and the $T's$.  More precisely,  
by (\ref{another conjugation}) we have   $U_1^l=\overline {(U_2^l)}, T_1^l=\overline{(T_2^l)}$. 

If we  consider the function $v_l$ to be defined as before, $v_l(z)=u_l(\overline z)$, then $v$ would be harmonic but the boundary conditions $0$ and $1$ would not correspond to the  same  boundary components as the original function $u$. In particular, one cannot use the uniqueness of solutions of the Dirichlet problem to conclude that $u_l$ and $v_l$ are the same.

To circumvent  this difficulty, in the  case where $p_1=\overline {p_2}$ we define a $\underline {\rm new}$ $v_l$ by 
\begin{eqnarray} v_l(z)=1-u_l(\overline z). \end{eqnarray}
 It is now clear that  $u_l$ and $v_l$ are harmonic and have the same boundary values. By uniqueness, $u_l=v_l$, and so 
 \begin{eqnarray} \label {second u} u_l(z)=1-u_l(\overline z). \end {eqnarray}
 In particular, if $z=\overline z$ lies in  the domain of $u$ one has  $u(z)=\frac{1}{2}$ and therefore  we are dealing with a  level line. Since $F^{-1}(l)$ is conformally a punctured sphere, $\nabla u_l$ never vanishes and is perpendicular to the above level line, whereas  $i\nabla u_l$ is tangent.
With this in mind, we define  a new $\xi_0$ by setting 
 \begin{eqnarray} \label{new x} \xi_0(l)=DF(p_3) i \nabla u_l(p_3)=iDF(p_3)  \nabla u_l(p_3), \; l\in \mathbb R \mathbb P^1. \end{eqnarray}
The second equality follows because $F$ is holomorphic. Repeating the arguments given after (\ref{first u}),  but using  (\ref{second u}) as the starting point instead of $u_l(z)=u_l(\overline z)$ , we obtain
\begin{eqnarray}\label {with a minus} \overline {\nabla u_{l}(z)}=-\nabla u_l(\overline z) \end{eqnarray}

To finish the proof, we must establish $\xi_0(l)=\overline {\xi_0(l)}$ whenever $l=\overline l$. To check this, use $p_3=\overline{p_3}$, (\ref{new x}) and  (\ref{with a minus}),  to see that $\overline {\xi_0(l)}$ equals
\begin{eqnarray*}
\overline{DF(p_3)i\nabla u_{ l}(p_3)}=(-i)\overline{DF(p_3)} \;\; \overline {\nabla u_{l}(p_3)}=(-i)DF(p_3)(-1)\nabla u_l(\overline{p_3})=iDF(p_3)\nabla u_l(p_3), 
\end{eqnarray*}
And so  $\xi_0(l)=\overline {\xi_0(l)}$ whenever $l=\overline l$, as desired. As before, this contradicts Lemma \ref{conjugation twice}, and  the proof of Theorem \ref{newJac} is finished.
\vskip10pt

\section {A conjecture in    approximation theory} 
\vskip10pt
\noindent In this section we make  some observations    concerning the Jacobian conjecture  and its relation to the  ideas discussed in this paper. 

In the absence of a proof that polynomial local biholomorphisms  are injective, one might want to  aim for a less ambitious goal, and try to settle the question at  the  ``next level" of difficulty.  Namely, one would like to show at least that,  under the assumptions of (JC),   the pre-images of \it all  complex lines  are connected. \rm 

This can be thought of as a ``vestigial"  form of injectivity, since   actual injectivity means that the pre-images of \it all points  are connected \rm (see  section 2). Note that,  as discussed in Example 5, subsection 2.5, the pre-images of all real hyperplanes are always connected in the (JC) setting.

However, even the  simpler problem of showing that \it all \rm complex lines in $\mathbb C^2$ in  a family of parallel lines pull-back into connected complex curves  seems  to be open at this time. More precisely, if $(f,g):\mathbb C^2 \to \mathbb C^2$ is as in the Jacobian conjecture   then,  to the best of our knowledge,  it is only known that the curve $f+c=0$ is connected for  \it all but at most finitely many \rm  values of  $c\in \mathbb C$. 

Equivalently, one would like to show at least that  if $(f,g)$ is any Jacobian pair, then $f$ and $g$ are irreducible polynomials. Given its  seemingly simpler nature,     solving this  last problem  in the affirmative  would be a reasonable litmus test for the veracity of the Jacobian Conjecture.

The techniques in this paper, meant  to deal with the less restrictive case of injectivity of local diffeomorphisms,  offer no indication of why (JC) should  hold. Be that as it may,  our  results  actually hint,  ever so gently,  at the possibility that (JC) might actually be false. Indeed,  the potential lack of connectedness of the pre-images of complex lines, and  the possibility that even if these complex curves are connected they may have positive genera, appear  as major  stumbling blocks. 

Various conjectures have been formulated over the years   whose validity would imply that of (JC).  Here, in view of the previous comments,  we go in the opposite direction by formulating a natural conjecture in approximation theory which, if true,    would imply that (JC)  $\underline {\rm fails}$, in principle without the onus of having to exhibit a counterexample.   The new viewpoint  is conceptually related to  the simple fact that  any holomorphic map 
$\mathbb C^n \to \mathbb C^n$  can be uniformly approximated,  over any given compact subset, by polynomial maps. 

We juxtapose  (JC) to  the new conjecture,   in case  the  reader feels inclined 
to  ponder  if  one statement is more likely to be true  than the other. 
\vskip10pt
\noindent \bf The Jacobian conjecture. \rm Every  polynomial map   $F:\mathbb C^n \to \mathbb C^n$, $n\geq1$,   with constant non-zero Jacobian determinant is injective.
\vskip10pt
\noindent \bf The polynomial approximation conjecture. \rm Every holomorphic map  $G:\mathbb C^n \to \mathbb C^n$, $n\geq 1$,  with constant non-zero Jacobian determinant  can be uniformly approximated over  compact subsets  by  polynomial maps   $F:\mathbb C^n \to \mathbb C^n$  with constant non-zero Jacobian determinant. 
\vskip10pt
When $n=1$ both  statements are trivially true. To see why the validity of the second conjecture,  for some  $n\geq 2$,   implies that (JC) is false,   observe that in finite dimensions if a sequence of injective local diffeomorphisms  converges  uniformly over compact subsets to a local diffeomorphism, then the limit map  is also injective (an easy argument using degree theory can be found in  \cite [p.411]{SX}). 

Since  the obviously non-injective map $G:\mathbb C^2 \to \mathbb C^2$,  $G(z_1,z_2)=(e^{z_1}, z_2e^{-z_1})$,  has Jacobian determinant one (similarly for $G(z_1, \dots, z_n)=(e^{z_1}, z_2 e^{-z_1}, z_3, \dots, z_n)$ when $n\geq 3$), the tail of  a hypothetical sequence of polynomial local  biholomorphisms converging  uniformly to $G$ over a ball containing a pair of points with the same image under $G$  would necessarily  contain non-injective maps, which would disprove (JC). For more on these ideas, see problem xv), section 11.

\vskip10pt

\section{Invertibility    in dynamics, differential geometry,  and complex analysis
}

We conclude this paper by drawing  attention to  some  problems  involving global injectivity, mostly from analysis, geometry and dynamics. 
The list below is far from exhaustive and, for the most part, we are leaving out  problems of an algebraic origin. For these, the reader should consult the extensive and ever evolving bibliography on the Jacobian conjecture.
\vskip20pt
\noindent i) (The Jacobian conjecture) Every polynomial local biholomorphism $\mathbb C^n \to \mathbb C^n$ is injective. 
\vskip10pt
\noindent  ii) (The weak Jacobian conjecture)  For  any given polynomial local biholomorphism $F:\mathbb C^n \to \mathbb C^n$  there is at least one 
$q\in \mathbb C^n$ that is covered precisely once by $F$.  This holds when $F$ is a   Dru\.{z}kowski-Jag\v{z}ev map; see Example 1,  section 2 ($q=0$ is covered only once). Can one work ``backwards" in the proofs of the reduction theorems  and conclude that every $F$   satisfying the hypotheses of (JC)  admits at least one fiber that is a singleton?
\vskip10pt
\noindent iii) The global asymptotic stability conjecture of Markus-Yamabe) and subsequent developments \cite{F}, \cite {AG}, \cite {CG}.
\vskip 10pt
\noindent iv)  The  geometrization of the Gutierrez  spectral inverse function theorem in terms of Cartan-Hadamard surfaces \cite{CG}),  \cite{X4}.
\vskip 10pt
\noindent v)  The Gale-Nikaid\^o conjecture states that if $Q^n$ is an $n$-rectangle in $\mathbb R^n$ and $F:Q^n \to Q^n$ is a local diffeomorphism for which the principal minors of $DF(x)$ do not vanish on $\partial Q^n$, then $F$ is injective. This type of problem comes up in mathematical economics, namely in the study of large scale supply-demand systems. See \cite {GN} for the original work, in two dimensions.
\vskip 15pt
\noindent vi)   The classification of  \it injective conformal harmonic maps $\mathbb R^2 \to \mathbb R^3$\rm.  The images of these maps are embedded simply-connected minimal surfaces. When complete,  these surfaces must be   planes and helicoids. 

This was first established for properly embedded surfaces by Meeks-Rosenberg \cite{MR}, using  of the Colding-Minicozzi theory of embedded minimal discs (the latter appeared in a series of papers in the Annals of Mathematics -- see, for instance,  \cite{CoMi}).

One expects that the conformal type being $\mathbb C$  alone (without the geometric assumption of completeness) should  be enough to characterize these minimal surfaces. 
Explicitly,  we propose:
\vskip5pt
\noindent \it Conjecture. \rm  The plane and the helicoid as the only  embedded minimal surfaces in $\mathbb R^3$ that are conformal to $\mathbb R^2$.
\vskip10pt

The scarcity of examples of embedded minimal surfaces conformal to $\mathbb R^2$  is reminiscent of  the fact that the functions $z \to az+b$, $a\neq 0$, and their conjugates,   are the only 
\it injective conformal harmonic maps $\mathbb R^2 \to \mathbb R^2$ \rm (there is some redundancy in the last statement due to the Cauchy-Riemann  equations).  Notice  the similarity between the  two  italicized sentences above,  and that the  dimensions of the targets went from 3 to 2. 

There are reasons to believe that there may be  a common thread between these  two statements. A Riemannian Bieberbach estimate \cite{FX} might be a starting point. Indeed, applying the classical Bieberbach estimate $|f''(0)|\leq 4|f'(0)|$, valid for injective holomorphic maps in the unit disc, to suitable scalings of an injective entire function $g$, one can  show that $g$ must be  an affine linear map. 

What is lacking in the study   of parabolic simply-connected  embedded minimal  surfaces is a well-chosen   application of \cite {FX}, together with a suitable scaling process that, in the limit, gives rise to only ruled minimal surfaces. By a classical theorem of Catalan, these are   planes and helicoids.
\vskip 10pt
\noindent vii)  The study  of $\text{Aut}(\mathbb C^n)$, the group of biholomorphisms of $\mathbb C^n$, when $n\geq 2$  \cite{AL}. One can view (JC) as part of this much larger question. 
\vskip 10pt

\noindent 
\noindent viii) The  rigidity of the identity in $\mathbb C^n$,  $n\geq 2$,   among injective entire maps \cite{X2}. This problem relates to vii). 
When $n=1$,  ``rigidity of the identity" (the statement $f:\mathbb C \to \mathbb C$ holomorphic and injective, $f(0)=0$, $f'(0)=1$, implies $f(z)=z$) leads to the description  of $\text {Aut} (\mathbb C)$ as the group of affine linear maps, via the argument sketched in vi). This opens up the possibility that a suitable strengthening  of  the rigidity result from \cite{X2} may lead to some kind of description, albeit a necessarily complicated one, of $\text{Aut}(\mathbb C^n)$ when $n\geq 2$.

In view of \cite{X2} we conjecture that any  biholomorphism $F \in \text{Aut} (\mathbb C^n)$ is   uniquely determined by only   two   data, one of which is of a global nature: Its  $1$-jet  at $0$   
and the Levi matrix (complex Hessian) of  
$$ \widehat F\stackrel{def}{=}\log ||DF(0)^{-1}(F-F(0))||.$$  
Formally,  we conjecture that  the map
$$\Lambda_n: \text{\rm Aut} (\mathbb C^n) \rightarrow \mathbb C^n \times GL(n, \mathbb C) \times \mathcal \rm{C^{\omega}}(\mathbb C^n - \{0\}, \mathbb C^{n\times n}),$$  where 
$\Lambda_n(F)=\big(F(0), DF(0), 
\big({ \widehat F _{{z_j}\overline{z}_k}} \big )\big )$,
is   injective for all $n\geq 1$.

 If  so,  the parametrization   $(\Lambda_n)^{-1}: \text{\rm Ran} (\Lambda_n) \to \text{\rm Aut} (\mathbb C^n)$  would exhibit $\text{\rm Ran} (\Lambda_n)$ only as a \it proto \rm moduli space (as opposed to a full-fledged moduli space)   since at present an intrinsic description 
of $\text{\rm Ran} (\Lambda_n)$ (one that does not involve $\Lambda_n$)  is only available when $n=1$. 
Indeed,  the plurisubharmonic   function $\widehat F$ is  actually  harmonic if $n=1$, and  so    $\text{\rm Ran} (\Lambda_1)=\mathbb C\times \mathbb C^{*}\times \{0\}$. Thus, $ \Lambda_1(F)=\Lambda_1(F(0)+DF(0))$, and by injectivity  $ F=F(0)+DF(0)$. In this way one recovers  one recovers  the standard description $ \text{\rm Aut} (\mathbb C)=\text{\rm Aff} (\mathbb C)$. 
 
\vskip 10pt
\noindent ix)  The Dixmier conjecture \cite{BK}.
\vskip 10pt
\noindent x)  Injectivity on metric spaces \cite {Gutu}-\cite{GGJ}.
\vskip 10pt
\noindent xi) Connections between  global invertibility and dynamical systems motivated by   Pinchuk's example \cite{Pi} of a non-injective polynomial local diffeomorphism of $\mathbb R^2$. For instance, in the partly expository paper \cite{MX} it is shown that a polynomial local diffeomorphism $F=(F_1, F_2):\mathbb R^2 \to \mathbb R^2$ 
is invertible if the polynomial 
$F_1\in \mathbb R[x, y]\subset \mathbb C[x, y]$ has one place at infinity. 

Explicitly, $F$ is injective provided that the (possibly singular) curve 
$\mathcal{C}=\{F_1=0\}\subset \mathbb C^2$ is irreducible, its  projectivization $\widetilde{\mathcal{C}}$ has only one point on the line at infinity, and the said  point is covered only once by a desingularization $\mathcal{R}\to \widetilde{\mathcal{C}}$. 

This theorem generalizes to the case of local diffeomorphisms a well-known application of the Abhyankar-Moh theory in algebraic geometry to the Jacobian conjecture in $\mathbb C^2$. Underlying the proof are some ideas motivated by the solution by C. Guti\'errez \cite{CG} of the Markus-Yamabe global stability conjecture in dynamical systems, in particular the notion of half-Reeb components.

In another theorem of  \cite{MX} it is shown that even if a polynomial local diffeomorphism $F:\mathbb R^2 \to \mathbb R^2$ is not globally injective (say, as in Pinchuk's example), for every $r>0$ there is a disc $B_r\subset \mathbb R^2$ of radius $r$ such that the restriction $F|B_r$ is injective. 
It would be of interest to  investigate under what conditions this property persists for polynomial self-maps of $\mathbb R^2$ which don't fail too badly the local diffeomorphism condition.
\vskip 10pt
\noindent xii) Geometric estimates related to invertibility (\cite{B2}, \cite {X1}).  Using  a dynamical interpretation,   and taking into account the closures of the Grassmannian-valued Gauss maps of submanifolds of $\mathbb R^n$, the results in \cite{X1} extend to  the non-linear setting  the familiar intersection properties of affine subspaces (a sort of ``non-linear version" of linear algebra). 

For instance, a special case of the main result of \cite{X1} states that two properly immersed  surfaces in $\mathbb R^3$ will intersect  provided that  the closures  in projective space of the images of their unoriented  Gauss maps do not intersect. It would be of interest to extend these results to metric spaces.  
\vskip10pt
 \noindent xiii) The possibility of establishing  infinite-dimensional versions, perhaps  for  Fredholm operators, of  certain invertibility phenomena (\cite{B1}, \cite{B2}, \cite{BK}, \cite{NX2},  \cite{R}, \cite{SX}, \cite{X1}). A deeper understanding might lead to new existence and uniqueness results for non-linear operators.
 \vskip10pt
 \noindent xiv)  It  was conjectured  in \cite{NX2}  that a local diffeomorphism $F:\mathbb  R^n \to  \mathbb  R^n$  must be injective if it has the property that  the pre-image of every    affine hyperplane is a connected set (compare with Example 4,  subsection 2.4).
 If true, this conjecture would imply the Jacobian conjecture, via a Bertini-type theorem (see Example 5,  subsection 2.5). 
 
 After the conclusion of this work we received a  very  interesting paper by Braun, Dias and Venato-Santos  \cite{B}, where the authors construct an explicit  counterexample to the above-mentioned conjecture from \cite{NX2}, namely  the map 
$$ F:\mathbb R^3 \to \mathbb R^3, \;F(x,y,z)=(z^5+e^x \cos y, z^3+e^x \sin y, z).$$  
 
As explained in \cite{B}, the pre-image of any  plane $\pi$ is either  homeomorphic to $\mathbb R^2 $ 
 (in fact, conformal to $\mathbb C$), or it is homeomorphic to $S^2$ with infinitely many points removed.  
 
 For a fixed $\pi $ with $F^{-1}(\pi)$ non-simply-connected, it is possible to excise a closed set (a union of curves) from $F^{-1}(\pi)$ so as to make the complement $1$-connected.  It would be of interest to examine the question of whether one can produce a  \it single  closed set \rm  $A\subset \mathbb R^3$  that  would cut   the pre-image of  \it every \rm plane through a base point  in a way that the complement (in the pre-image) of the cut is  $1$-connected.

 More precisely, one seeks $q\in F(\mathbb R^3)$ with  $\#F^{-1}(q)>1$, and a closed set $A\subset \mathbb R^n$ such that $F^{-1}(q)\cap A=\emptyset$
and  $F^{-1}(\pi)- F^{-1}(\pi)\cap A$ is  $1$-connected  for every plane $\pi$ passing through  $q$.    

If this can be arranged,  the pre-image of every plane $\pi$ through $q$ under the map $\mathbb R^3- A\stackrel{F}{\rightarrow} \mathbb R^3$ would be topologically a plane and yet $\#F^{-1}(q)>1$.  This would show that the parabolicity hypothesis  is essential in Corollary \ref{injectivity}.  Since, in the above construction,  the prei-mage of a plane  is made simply-connected by means of cuts, it is indeed natural that the conformal type of the pre-image  be that of the open unit disc in $\mathbb C$, rather than $\mathbb C$ itself, in accordance with Corollary \ref{injectivity}.

\vskip 10pt
\noindent xv) Let $\mathcal P$  denote the family of all polynomial self-maps of $\mathbb C^n$ into itself with  constant Jacobian determinant $1$. The Jacobian conjecture claims that $\mathcal P \subset \text{Aut} (\mathbb C^n)$.  Here, we consider the closure of $\mathcal P$ relative to the natural topology   on $H(\mathbb C^n, \mathbb C^n)\subset C(\mathbb C^n, \mathbb C^n)$ that  corresponds to uniform convergence on compact subsets.   It follows from the Baire category theorem that $\mathcal P$ is not closed if $n\geq 2$.  What can then be said about   the  set 
$\overline{\mathcal P}-\mathcal P$ ? By the topological degree argument from section 10, if one assumes the validity of  (JC)  then $\overline{\mathcal P}$  contains only injective local biholomorphisms. In particular, one would like to know if  $\overline{\mathcal P}-\mathcal P$  contains the classical Fatou-Bieberbach example of an injective entire self-map of 
$\mathbb C^2$  whose image omits an open ball.

\vskip20pt
\noindent \text{Frederico Xavier}\\ }
 \noindent \text{Department of Mathematics}\\
 \text{Texas Christian University}\\
 \text{f.j.xavier@tcu.edu}\\

\end{document}